\documentclass[11pt,a4paper,reqno]{amsart}
\usepackage{amsfonts}
\usepackage{amsmath,amssymb,amsthm,amsxtra}
\usepackage{float}
\usepackage[colorlinks, linkcolor=blue,anchorcolor=Periwinkle,
    citecolor=red,urlcolor=Emerald]{hyperref}
\usepackage[usenames,dvipsnames]{xcolor}
\usepackage{enumitem}
\usepackage{geometry,array} 
\usepackage{graphicx}
\usepackage{subfigure}
\usepackage{bookmark}
\usepackage{tikz}\usetikzlibrary{matrix}
\usepackage{url}

\usetikzlibrary{decorations.markings}
\tikzset{->-/.style={decoration={  markings,  mark=at position #1 with
    {\arrow{>}}},postaction={decorate}}}
\tikzset{-<-/.style={decoration={  markings,  mark=at position #1 with
    {\arrow{<}}},postaction={decorate}}}
\usepackage[all]{xy}
\usetikzlibrary{arrows}
\usetikzlibrary{decorations.pathreplacing,decorations.pathmorphing,shadings,fadings,calc}

\theoremstyle{plain}
\newtheorem{theorem}{Theorem}[section]
\newtheorem*{thm}{Theorem~1}

\newtheorem{lemma}[theorem]{Lemma}
\newtheorem{corollary}[theorem]{Corollary}
\newtheorem{proposition}[theorem]{Proposition}
\newtheorem{conjecture}[theorem]{Conjecture}
\theoremstyle{definition}
\newtheorem{definition}[theorem]{Definition}

\newtheorem{example}[theorem]{Example}

\newtheorem{remark}[theorem]{Remark}
\newtheorem{notations}[theorem]{Notations}
\numberwithin{equation}{section}

\def\hua{\mathcal}
\def\kong{\mathbb}
\def\<{\langle}
\def\>{\rangle}
\def\NN{\mathbb{N}}
\def\ZZ{\mathbb{Z}}

\def\Aut{\operatorname{Aut}}

\def\Sim{\operatorname{Sim}}
\def\Hom{\operatorname{Hom}}
\def\End{\operatorname{End}}
\def\Ext{\operatorname{Ext}}

\def\Stab{\operatorname{Stab}}
\def\Stap{\operatorname{Stab}^\circ}
\def\diff{\operatorname{d}}
\def\Br{\operatorname{Br}}

\def\deg{\operatorname{deg}}
\newcommand{\h}{\operatorname{\hua{H}}}            

\renewcommand{\k}{\mathbf{k}}
\renewcommand{\mod}{\operatorname{mod}}

\newcommand{\tilt}[3]{{#1}^{#2}_{#3}}
\newcommand{\Cone}{\operatorname{Cone}}

\def\numbers{\begin{enumerate}[label=\arabic*{$^\circ$}.]}
\def\ends{\end{enumerate}}

\newcommand{\EG}{\operatorname{EG}}       
\newcommand{\EGp}{\operatorname{EG}^\circ}       
\newcommand{\SEGp}{\operatorname{SEG}^\circ}       
\newcommand{\C}{\hua{C}}

\newcommand{\CEG}[2]{\operatorname{CEG}_{#1}(#2)}             
\newcommand{\D}{\operatorname{\hua{D}}}

\newcommand{\per}{\operatorname{per}}
\newcommand\Sph{\operatorname{Sph}}
\def\zero{\hua{H}_\Gamma}
\newcommand{\Tri}{\bigtriangleup}

\def\arrow{red}
\def\surf{\mathbf{S}}                       

\newcommand{\ST}{\operatorname{ST}}        
\newcommand{\BT}{\operatorname{BT}}        

\newcommand{\MCG}{\operatorname{MCG}}

\newcommand{\Int}{\operatorname{Int}}
\newcommand\Dehn[1]{\operatorname{D}_{#1}}

\def\TT{\mathbf{T}}
\def\T{\kong{T}}

\def\M{\mathbf{M}}

\newcommand{\AT}{\operatorname{AT}}        
\def\RHom{\operatorname{RHom}}
\def\ee{\operatorname{\mathfrak{E}}}
\def\EE{\operatorname{\hua{E}}}

\def\surfo{{\mathbf{S}}_\Tri}
\def\cA{\operatorname{CA}}
\def\CA{\overline{\cA}}

\def\oAp{\operatorname{OA}^\circ}

\newcommand{\Q}[1]{\mathcal{Q}(#1)}
\newcommand\Bt[1]{\operatorname{B}_{#1}}
\newcommand\bt[1]{\operatorname{B}_{#1}^{-1}}
\def\ff{\operatorname{\mathrel{\Big|}}}

\newcommand{\MMCG}{\operatorname{MMCG}}
\newcommand{\Quad}{\operatorname{Quad}}
\newcommand{\Conf}{\operatorname{Conf}}

\title[Decorated marked surfaces]
{Decorated marked surfaces:\\ spherical twists versus braid twists}
\author{Yu Qiu}
\address{Yu Qiu, Institutt for matematiske fag, NTNU, N-7491 Trondheim, Norway.}
\email{Yu.Qiu@Bath.edu}
\date{\today}
\begin{document}

\begin{abstract}
    We are interested in the 3-Calabi-Yau categories $\D$ arising from quivers with potential
    associated to a triangulated marked surface $\surf$ (without punctures).
    We prove that the spherical twist group $\ST$ of $\D$ 
    is isomorphic to a subgroup (generated by braid twists) 
    of the mapping class group of the decorated marked surface $\surfo$.
    Here $\surfo$ is the surface obtained from $\surf$
    by decorating with a set of points,
    where the number of points equals the number of triangles in any triangulations
    of $\surf$.
    For instance, when $\surf$ is an annulus,
    the result implies that the corresponding spaces of stability conditions on $\D$
    are contractible.

    \vskip .3cm
    {\parindent =0pt
    \it Key words:} Calabi-Yau categories, spherical twists, quivers with potential,
    braid group, stability conditions

\end{abstract}
\maketitle
\tableofcontents\addtocontents{toc}{\setcounter{tocdepth}{1}}


\section{Introduction}

\subsection{Calabi-Yau~(CY) categories from mirror symmetry}
We are interested in a class of 3-Calabi-Yau categories $\D$
arising from (homological) mirror symmetry.
These 3-CY categories are not only interesting in mathematics (\cite{KS} and \cite{ST}),
but also in string theory (\cite{GMN}, cf. \cite{BS}).
On the symplectic geometry side, the category $\D$ (of type A) was first studied
by Khovanov-Seidel \cite{KS}.
They showed that there is a faithful braid group action on $\D$.
Moreover, when realizing $\D$ as the subcategory of the derived Fukaya category of
the Milnor fibre of a simple singularities of type A,
such a braid group is generated by the (higher) Dehn twists along certain Lagrangian spheres.
On the algebraic geometry side,
Seidel-Thomas \cite{ST} studied the mirror counterpart of \cite{KS} (also in type A).
They showed that $\D$ can be realized as a subcategory of the bounded derived category of
coherent sheaves of the mirror variety with a faithful braid group action.
Recently, Smith \cite{S} showed that if $\D$ is coming from triangulations of marked surfaces $\surf$,
then it also can be embedded into some derived Fukaya category.
This class of cases are the ones we will study.
Our focus is on the spherical twist group $\ST\subset\Aut\D$,
a subgroup of the auto-equivalence group of $\D$
generated by Khovanov-Seidel-Thomas~(KST) spherical twists.
The aim is to generalize KST's result,
that $\ST$ is `faithful',
in the sense that $\ST$ is isomorphic to the classical (type A) braid group
(and in general, isomorphic to a subgroup of a certain mapping class group).
We need to restrict ourselves in the case when marked surfaces are unpunctured.
In the twin paper \cite{KQ1}, we will make an effort to attack the problem when the marked surfaces are punctured.

Note that the spherical twist group $\ST$
acts freely on the space $\Stap\D$ of Bridgeland's stability condition of $\D$.
This is one of our main motivations to study such a group.
In fact, Bridgeland-Smith~(BS) \cite{BS} recently
showed that the quotient (orbifold) $\Stap\D/\Aut^\circ$ is isomorphic to
the moduli space $\Quad_\heartsuit(\surf)$ of meromorphic quadratic differentials
with simple zeroes on the marked surfaces $\surf$,
where $\Aut^\circ\D$ is the extension of the (tagged) mapping class group of $\surf$ on top of $\ST$.
And one would expect that the faithfulness of spherical twist group actions will imply
the simply connectedness of $\Stap\D$.
For instance, this implication holds for the (3-CY) Dynkin case (see \cite{Q2});
also, such faithfulness (and its implication of simply connectedness)
was proved by Brav-Thomas \cite{BT} for the 2-CY Dynkin case
and by Ishii-Ueda-Uehara \cite{IUU} for the 2-CY affine $\widetilde{A}$ case.

Our main result says that $\ST$ is isomorphic to a subgroup of
the mapping class group of some surface.
As an example, we will show the contractibility of the corresponding $\Stap\D$ in this paper.
In the sequel,
we will prove that this result indeed
implies the simply connectedness of $\Stap\D$ for any unpuncutred marked surface $\surf$.

\subsection{Quivers with potential and categorification of cluster algebras}
Quiver mutation was invented by Fomin-Zelevinsky~(FZ) around 2000,
as the combinatorial aspect of cluster algebras.
Later, mutation was developed by Derksen-Weyman-Zelevinsky~(DWZ) for quivers with potential.

The first (additive) categorification of cluster algebras (with certain associated acyclic quivers)
was due to Buan-Marsh-Reineke-Reiten-Todorov,
via representations of the corresponding quivers.
Amiot introduced the generalized cluster categories via Ginzburg dg algebras
for quivers with potential.
In her construction, the cluster category $\C(\Gamma)$ is defined by the following
short exact sequence of triangulated categories
\begin{equation}\label{SES}
    0 \to \D_{fd}(\Gamma) \to \per\Gamma  \xrightarrow{\pi} \C(\Gamma) \to 0,
\end{equation}
where $\Gamma=\Gamma(Q,W)$ is the Ginzburg dg algebra of the quiver with potential and
$\per\Gamma$ (resp. $\D_{fd}(\Gamma)$) are the perfect (resp. finite-dimensional)
derived category of $\Gamma$.
Here, $\D_{fd}(\Gamma)$ is the 3-CY category we mentioned above
and it also provides a categorification for cluster algebras.

There is an exchange graph associated to each of the categories in \eqref{SES}, namely:
\begin{itemize}
  \item the reachable hearts/t-structures in $\D_{fd}(\Gamma)$ as vertices
  and simple tilting as edges for the exchange graph $\EGp(\D_{fd}(\Gamma))$;
  \item the reachable silting sets in $\per\Gamma$ as vertices
  and mutation as edges for the silting exchange graph $\SEGp(\per(\Gamma))$;
  \item the cluster tilting sets in $\C(\Gamma)$ as vertices
  and mutation as edges for the cluster exchange graph $\CEG{}{\C(\Gamma)}$.
\end{itemize}
They play a crucial role in categorifying cluster algebras,
understanding quantum dilogarithm identities and computing stability conditions.
By simple-projective duality,
there is a canonical isomorphism between the first two graphs.
Moreover, they are coverings of the third (cf. \cite{KQ})
by the spherical twist group action we mentioned above.

\subsection{Triangulations of marked surfaces}
A geometric aspect of cluster theory was explored by Fomin-Shapiro-Thurston~(FST).
They constructed a quiver $Q_\TT$ for each (tagged) triangulation $\TT$
of a marked surface $\surf$ and showed that
flipping triangulations corresponds to FZ mutation of quivers.
Here, the marked surface $\surf$ is a surface with marked points on its boundaries and
punctures in its interior.
Further, Labardini-Fragoso gave a rigid potential $W_\TT$ for each FST quiver $Q_\TT$,
which is the unique `good' (rigid, to be precise) one (cf. \cite{GLFS}), that is compatible with DWZ mutation.
Then one can construct the Ginzburg dg algebra $\Gamma_\TT=\Gamma(Q_\TT,W_\TT)$
and the associated categories, as in \eqref{SES}.

In this paper, we will deal the case when $\surf$ is unpunctured
and introduce a new surface from $\surf$ by decorating it with a set $\Tri$ of points
as a topological model for these categories.
The number of points in $\Tri$ equals the number of triangles in
any triangulation of $\surf$.
This decorating idea already appeared in various contexts
(e.g. Krammer \cite{Kr} and Gaiotto-Moore-Neitzke \cite{GMN}).
In the theory of BS (\cite{BS}),
these decorating points are simple zeroes of quadratic differentials (cf. Figure~\ref{fig:WH});
the boundary components of $\surf$ are the real blow-up of higher order ($\geq 3$) poles of quadratic differentials.
Further, when considering the mapping class group of $\surfo$, these decorating points
are serving as punctures in topology; however, we reserve the terminology `punctures'
for the FST setting of marked surfaces.

Denote such a surface by $\surfo$ and call it the \emph{decorated marked surface}.
A triangulation of $\surfo$ is a maximal collection of simple open arcs that divides
$\surfo$ into triangles such that each triangle contains exactly one decorating point.
One important feature of $\surfo$ is that flipping a triangulation has directions (cf. \S~\ref{sec:Kra}).
Then we obtain a list of correspondences, as shown in Table~\ref{table}
(some of the correspondences will be given in the second part of the paper).
Simple closed arcs, i.e. the simple arcs connecting different decorating points, play a crucial role
in the construction/proof of these correspondences.
In the theory of BS, they should correspond to stable objects
(w.r.t. some stability conditions)
and saddle connections (w.r.t. some quadratic differentials).

\begin{table}[ht]
\caption{Correspondences}
\label{table}
\setlength{\extrarowheight}{2pt}
\begin{tabular}{ccc}
\\Topological side&&Categorical side\\[1ex]
\hline
Braid twists&$\xymatrix@C=2pc{\ar[r]^{\cong}&}$&Spherical twists\\
\hline\hline
Simple closed arcs in $\surfo$&$\xymatrix@C=2pc{\ar[r]^{\text{\tiny{1-1}}}_{\text{\tiny{up to $[1]$}}}&}$
&Spherical obj. in $\D_{fd}(\Gamma_\TT)$\\
\footnotesize{Dual Tri. with Whitehead moves} &&
\footnotesize{Hearts with simple tilting}
\\ \hline\hline
\text{\tiny{graph dual}}$\;\;\,\mathrel{\Bigg\updownarrow}\qquad\quad\;\;$
&&$\;\;\;\qquad\qquad\mathrel{\Bigg\updownarrow}\;\;\;$\text{\tiny{sim.-proj. dual}}
\\ \hline\hline
Reachable open arcs in $\surfo$&$\xymatrix@C=2pc{\ar[r]^{\text{\tiny{1-1}}}&}$&
Reachable ind. in $\per\Gamma_\TT$\\
\footnotesize{Triangulations with flips}&&
\footnotesize{Silting with mutation}\\
\hline\hline
\text{\tiny{forgetful map}}
$\begin{smallmatrix}\\\surfo\\\mathrel{\bigg\downarrow}\\ \surf\end{smallmatrix}
\qquad\qquad$
&&
\text{\tiny{}}$\qquad\qquad
\begin{smallmatrix}\\\per\Gamma_\TT\\\mathrel{\bigg\downarrow}\\ \C_\surf\end{smallmatrix}$
\text{\tiny{quotient map}}\\[4ex]
\hline\hline
Open arcs in $\surf$&$\xymatrix@C=2pc{\ar[r]^{\text{\tiny{1-1}}}&}$&
Rigid ind. in $\C_\surf$\\
\footnotesize{Triangulations with flips}&&
\footnotesize{Cluster tilting with mutation}\\
\hline\hline
\end{tabular}
\end{table}

\subsection{The project: decorated marked surfaces}
This paper initiates a project: \textbf{DMS}=decorated marked surfaces.
In the first paper, we prove the following theorem.
\begin{thm}
Suppose $\surf$ is a marked surface without punctures and $\TT$ a triangulation of
its decorated version $\surfo$.
There is a canonical isomorphism
\begin{gather}\label{1}
    \iota\colon\BT(\TT)\to\ST(\Gamma_\TT),
\end{gather}
sending the standard generators (i.e. braid twists of the closed arcs $\eta$ in the dual $\TT^*$)
to the standard generators (i.e. spherical twists of the corresponding spherical objects $X_\eta$).
\end{thm}

The topics/plan for the sequels are:
\begin{description}
\item[DMS (Part B)] We give a geometric realization of
    silting objects in $\per(\Gamma_\TT)$, simple-projective duality for $\Gamma_\TT$ and
    Amoit's quotient $\pi$ in \eqref{SES} that defines cluster categories.
\item[DMS II] We prove Conjecture~\ref{con0} and Conjecture~\ref{con1},
that the dimensions of homomorphisms between objects in $\D(\Gamma)$ equals
the intersection numbers between the corresponding arcs in $\surfo$.
This is a joint work with Yu~Zhou.
\item[DMS III] We show that there is a unique canonical way to identify
$\D(\Gamma_{\TT})$, for any triangulation $\TT$ in $\EGp(\surfo)$.
Thus, one can associate a unique 3-Calabi-Yau category $\D_{fd}(\surfo)$ to $\surfo$.
As an application, we show that the spherical twist group $\ST(\surfo)$ acts faithfully on
the corresponding space $\Stap\D_{fd}(\surfo)$ of stability conditions.
This is a joint work with Aslak~Buan.
\end{description}
We will prove the simply connectedness of $\Stap\D_{fd}(\surfo)$
by calculating the fundamental group of the space $\Quad(\surf)$ of quadratic differentials
in \cite{KQ1}.
\subsection*{Acknowledgements}
This work was inspired during joint working with Alastair King on the twin paper \cite{KQ1},
which deals with punctured marked surfaces.
I would like to thank my collaborators mentioned above, as well as
Tom~Bridgeland, Ivan~Smith, Dong~Yang, Idun Reiten and Bernhard~Keller
for inspiring conversations.

\section{Preliminaries}\label{sec:bg}
\subsection{Quivers with potential and Ginzburg algebras}\label{sec:QP}
Fix an algebraically closed field $\k$ and all categories are $\k$-linear.
Denote by $\Gamma=\Gamma(Q,W)$ the \emph{Ginzburg dg algebra (of degree 3)} associated to
a quiver with potential $(Q,W)$,
which is constructed as follows (cf. \cite{KY}):
\begin{itemize}
\item   Let $Q^3$ be the graded quiver whose vertex set is $Q_0$
and whose arrows are:
\begin{itemize}
\item the arrows in $Q_1$ with degree $0$;
\item an arrow $a^*:j\to i$ with degree $-1$ for each arrow $a:i\to j$ in $Q_1$;
\item a loop $e_i^*:i\to i$ with degree $-2$ for each vertex $i$ in $Q_0$.
\end{itemize}
\item The underlying graded algebra of $\Gamma(Q,W)$ is the completion of
the graded path algebra $\k Q^3$ in the category of graded vector spaces
w.r.t. the ideal generated by the arrows of $Q^3$.
\item The differential of $\Gamma(Q,W)$ is the unique continuous linear endomorphism,
homogeneous of degree $1$, which satisfies the Leibniz rule and takes the following values
\begin{itemize}
  \item $\diff a = 0$ for any $a\in Q_1$,
  \item $\diff a^* = \partial_a W$ for any $a\in Q_1$ and
  \item $\diff \sum_{e\in Q_0} e^*=\sum_{a\in Q_1} \, [a,a^*]$.
\end{itemize}
\end{itemize}

\begin{example}\label{ex:1}
Let $Q$ be a 3-cycle with edges $a,b,c$ and the potential $W=abc$.
Then the (graded) quiver $Q^3$ is
\begin{gather}\label{eq:ex1}
\xymatrix{\\Q^3\colon\\\\}
    \xymatrix@R=3pc@C=2.3pc{ &2  \ar@(ur,ul)[]_{f_2}
    \ar@<.8ex>[dr]^{a}\ar@<.2ex>[dl]^{b^*}
      \\ 1\ar@(l,d)[]_{f_1}\ar@<.2ex>[rr]^{c^*}\ar@<.8ex>[ur]^{b}&&
        3\ar@(d,r)[]_{f_3}\ar@<.2ex>[ul]^{a^*}
          \ar@<.8ex>[ll]^{c}
  }
\end{gather}
and the (non-trivial) differentials are
\begin{equation}
\begin{array}{c}
  d(a^*)=bc,\; d(b^*)=ca,\;d(c^*)=ab,\; \\
  d(f_1)=cc^*-b^*b,\;d(f_2)=bb^*-a^*a,\;d(f_1)=aa^*-c^*c.
\end{array}\label{eq:ex}
\end{equation}
\end{example}

In this paper, the quivers with potential we are considering
are \emph{rigid} (and hence \emph{non-degenerated}),
which basically means that they behave nicely under mutation,
in the sense of DWZ.
For details about these notions, see, e.g. \cite{KY} and \cite{GLFS}.

\subsection{The 3-Calabi-Yau categories}\label{sec:DC}
A triangulated category $\D$ is called \emph{$N$-Calabi-Yau} ($N$-CY)
if, for any objects $X,X'$ in $\hua{D}$ we have a natural isomorphism
\begin{gather}\label{eq:serre}
    \mathfrak{S}:\Hom_{\hua{D}}^{\bullet}(X,X')
        \xrightarrow{\sim}\Hom_{\hua{D}}^{\bullet}(X',X)^\vee[N].
\end{gather}
Note that the graded dual of a graded $\k$-vector space
$V=\oplus_{k\in\ZZ} V_k[k]$ is
\[V^\vee=\bigoplus_{k\in\ZZ} V_k^*[-k].\]
Further, an object $S$ is \emph{$N$-spherical} if $\Hom^{\bullet}(S, S)=\k \oplus \k[-N]$
and \eqref{eq:serre} holds functorially for $X=S$ and $X'$ in $\D$.
Denote by $\D_{fd}(\Gamma)$ the finite-dimensional derived category of $\Gamma$.
It is well-known that this is a 3-CY category.
We also know that (see, e.g. \cite{KN}) $\D_{fd}(\Gamma)$ admits a canonical heart $\zero$ generated
by simple $\Gamma$-modules $S_e$, for $e\in Q_0$, each of which is 3-spherical.
Recall that the \emph{twist functor $\phi$ of a spherical object} $S$
is defined by
\begin{gather}\label{eq:phi}
    \phi_S(X)=\Cone\left(S\otimes\Hom^\bullet(S,X)\to X\right)
\end{gather}
with inverse
\[
    \phi_S^{-1}(X)=\Cone\left(X\to S\otimes\Hom^\bullet(X,S)^\vee \right)[-1]
\]
Denote by $\ST(\Gamma)$ the spherical twist group of $\D_{fd}(\Gamma)$
in $\Aut\D_{fd}(\Gamma)$, generated by $\{\phi_{S_e}\mid e\in Q_0\}$.
By \cite[Lemma~2.11]{ST}, we have the formula
\begin{equation}\label{eq:212}
    \phi_{\psi(S)}=\psi\circ\phi_{S}\circ\psi^{-1}
\end{equation}
for any spherical object $S$ and $\psi\in\Aut\D_{fd}(\Gamma)$.

Denote by $\Sph(\Gamma)$ the set of \emph{reachable} spherical objects
in $\D_{fd}(\Gamma)$, that is,
\begin{gather}\label{eq:sph=st}
    \Sph(\Gamma)=\ST(\Gamma)\cdot\Sim\zero,
\end{gather}
where $\Sim\h$ denotes the set of simples of an abelian category $\h$.

We have the following observations.
\begin{itemize}
  \item The the twist functor is well-defined on $\Sph(\Gamma)/[1]$, i.e. $\phi_S=\phi_{S[1]}$.
  \item 
  Clearly, for any $\phi\in\ST(\Gamma)$ and $X\in\Sph(\Gamma)$,
  $\phi(X)$ is still in $\Sph(\Gamma)$.
  \item By \eqref{eq:212}, $\ST(\Gamma)$ is also generated by all $\phi_X$ for $X\in\Sph(\Gamma)$ (cf. \cite{KQ}).
\end{itemize}

\begin{remark}\label{def:iso}
Two elements $\psi$ and $\psi'$ in $\Aut\D_{fd}(\Gamma)$ are \emph{isotopic},
denote by $\psi\sim\psi'$,
if $\psi^{-1}\circ\psi'$ acts trivially on $\Sph(\Gamma)$.
In this paper, we will only consider the auto-equivalences up to isotopy,
i.e. we will consider $\ST(\Gamma)$ as a subgroup of
\begin{gather}\label{eq:Aut0}
    \Aut^\circ\D_{fd}(\Gamma)=\Aut\D_{fd}(\Gamma)/\sim.
\end{gather}
However, we will show in the sequel that: the identity is the only spherical twist
which acts trivially on $\Sph(\Gamma)$ in our case.
\end{remark}

\subsection{Triangulations of marked surfaces}\label{sec:MS}
Throughout the paper,
$\surf$ denotes a \emph{marked surface} without punctures in the sense of \cite{FST},
that is, a connected surface with a fixed orientation and
a finite set $\M$ of marked point on the (non-empty) boundary $\partial\surf$
satisfying that 
each connected component of $\partial\surf$ contains at least one marked point.
Up to homeomorphism, $\surf$ is determined by the following data
\begin{itemize}
\item the genus $g$;
\item the number $|\partial\surf|$ of boundary components;
\item the integer partition of $|\M|$ into $|\partial\surf|$
parts describing the number of marked points
on its boundary.
\end{itemize}
As in \cite[p5]{FST}, we will exclude the case when
there is no triangulation or there is no arcs in the triangulation.
In other words, we require $n\geq1$ in \eqref{eq:n}.

An (open) \emph{arc} in $\surf$ is a curve (up to homotopy) that connects two marked points in $\M$,
which is neither isotopic to a boundary segment nor to a point.
The \emph{intersection number} is defined to be
\[
    \Int(\gamma_1,\gamma_2)=
    \min\{ |\gamma_1'\cap\gamma_2'\cap(\surf-\M)| \ff  \gamma_i\sim\gamma_i' \}.
\]
An \emph{(ideal) triangulation} $\T$ of $\surf$
is a maximal collection of compatible simple arcs.
Here, compatible means any two arcs in $\T$ that do not intersect.
Moreover, it is well-known that any triangulation $\T$ of $\surf$ consists of
\begin{gather}\label{eq:n}
n=6g+3|\partial\surf|+|\M|-6
\end{gather}
(simple) arcs and divides $\surf$ into
\begin{gather}\label{eq:Tri}
    \aleph=\frac{2n+|\M|}{3}
\end{gather}
triangles.
Denote by $\EG(\surf)$ the \emph{exchange graph} of triangulations of $\surf$,
that is, the unoriented graph whose vertices are triangulation of $\surf$
and whose edges correspond to flips (see the lower pictures in Figure~\ref{fig:flip} for a flip).
It is known that $\EG(\surf)$ is connected.
If $\surf$ is an $(n+3)$-gon, then $\EG(\surf)$ is the associahedron of dimension $n$
(cf. Figure~\ref{fig:Pent}).

\begin{figure}
\begin{tikzpicture}[xscale=.4,yscale=.5]
  \foreach \j in {0, 8, 16}
  {\draw[thick](0+\j,0)to(2+\j,0)to(3+\j,2)to(1+\j,3)to(-1+\j,2)to(0+\j,0);}
  \foreach \j in {4, 12}
  {\draw[thick](0+\j,0-6)to(2+\j,0-6)to(3+\j,2-6)to(1+\j,3-6)to(-1+\j,2-6)to(0+\j,0-6);}
  \foreach \j in {0, 8, 16}
  {
    \draw[NavyBlue](0+\j,0)node{\footnotesize{$\bullet$}}
        (2+\j,0)node{\footnotesize{$\bullet$}}(3+\j,2)node{\footnotesize{$\bullet$}}
    (1+\j,3)node{\footnotesize{$\bullet$}}(-1+\j,2)node{\footnotesize{$\bullet$}}(0+\j,0);
  }
  \foreach \j in {4, 12}
  {
    \draw[NavyBlue](0+\j,0-6)node{\footnotesize{$\bullet$}}
        (2+\j,0-6)node{\footnotesize{$\bullet$}}(3+\j,2-6)node{\footnotesize{$\bullet$}}
    (1+\j,3-6)node{\footnotesize{$\bullet$}}(-1+\j,2-6)node{\footnotesize{$\bullet$}}(0+\j,0-6);
  }
\draw[NavyBlue,thick](0,0)to(1,3)to(2,0) (0,4.5)node{};
\draw[NavyBlue,thick](-1+8,2)to(2+8,0)to(1+8,3);
\draw[NavyBlue,thick](2+8+8,0)to(-1+8+8,2)to(3+8+8,2);
\draw[NavyBlue,thick](1+4,3-6)to(0+4,0-6)to(3+4,2-6);
\draw[NavyBlue,thick](0+12,0-6)to(3+12,2-6)to(-1+12,2-6);

\draw(1+4-1.5,1.5)to(1+4+1.5,1.5);
\draw(1+12-1.5,1.5)to(1+12+1.5,1.5);
\draw(1+2-1,1-3+1.5)to(1+2+.5,1-3-1);
\draw(1+6+8-.5,1-3-1)to(1+6+8+1,1-3+1.5);
\draw(1+8-1.5,1.5-6)to(1+8+1.5,1.5-6);
\end{tikzpicture}
\caption{The exchange graph of triangulations of a pentagon}
\label{fig:Pent}
\end{figure}
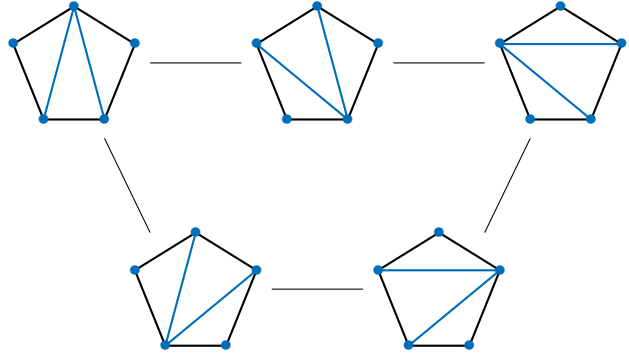

Let $\surf$ be a marked surface and $\T$ a triangulation of $\surf$.
Then there is an associated quiver $Q_\T$ with a potential $W_\T$, constructed as follows
(See, e.g. \cite{GLFS} or \cite{QZ} for the precise definition):
\begin{itemize}
\item the vertices of $Q_\T$ are (indexed by) the arcs in $\T$;
\item for each triangle $T$ in $\T$, there are three arrows
    between the corresponding vertices as shown in Figure~\ref{fig:quiver};
\item these three arrows form a 3-cycle in $Q_\T$ and
$W_\T$ is the sum of all such 3-cycles.
\end{itemize}
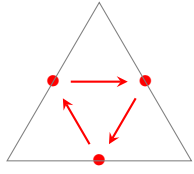
\begin{figure}[ht]\centering
  \begin{tikzpicture}[scale=.7]
  \foreach \j in {1,...,3}  { \draw (120*\j-30:2) coordinate (v\j);}
    \path (v1)--(v2) coordinate[pos=0.5] (x3)
              --(v3) coordinate[pos=0.5] (x1)
              --(v1) coordinate[pos=0.5] (x2);
    \foreach \j in {1,...,3}{\draw (x\j) node[red] (x\j){$\bullet$};}
    \draw[->,>=stealth,red,thick] (x1) to (x3);
    \draw[->,>=stealth,red,thick] (x3) to (x2);
    \draw[->,>=stealth,red,thick] (x2) to (x1);
    \draw[gray,thin] (v1)--(v2)--(v3)--cycle;
  \end{tikzpicture}
\caption{The (sub-)quiver associated to a triangle (with a potential)}
\label{fig:quiver}
\end{figure}

\section{Triangulations of decorated marked surfaces}\label{sec:3}
\subsection{Decorated marked surfaces}
Recall that any triangulation of $\surf$ consists of $\aleph$ triangles,
where $\aleph$ is given by the formula \eqref{eq:Tri}.
\begin{definition}\label{def:arcs}
The \emph{decorated marked surface} $\surfo$ is a marked surface $\surf$ together with
a fixed set $\Tri$ of $\aleph$ `decorating' points (in the interior of $\surf$,
where $\aleph$ is defined in \eqref{eq:Tri}),
which serve as punctures.
Moreover,
\begin{itemize}
\item An \emph{open arc} in $\surfo$ is (the isotopy class of) a curve in $\surfo-\Tri$
that connects two marked points in $\M$, which is neither isotopic to a boundary segment nor to a point.
\item a \emph{closed arc} in $\surfo$ is (the isotopy class of) a curve in $\surfo-\Tri$
that connects different decorating points in $\Tri$.
Denote by $\cA(\surfo)$ the set of simple closed arcs.
\item An \emph{L-arc} $\eta$ in $\surfo$ is (the isotopy class of) a curve in $\surfo-\Tri$
such that its endpoints coincide at a decorating point in $\Tri$ and it is not isotopic to a point.
\item A \emph{general closed arc} in $\surfo$ is
either a closed arc or an L-arc;
denote by $\CA(\surfo)$ the set of simple general closed arcs.
\end{itemize}
\end{definition}
The \emph{intersection numbers} between arcs in $\surfo$ are defined as follows:
\begin{itemize}
\item For an open arc $\gamma$ and any arc $\eta$,
their intersection number is the geometric intersection number in $\surfo-\M$:
\[
    \Int(\gamma,\eta)=\min\{ |\gamma'\cap\eta'\cap(\surfo-\M)|
      \ff \gamma'\sim\gamma,\eta'\sim\eta  \}.
\]
\item For two general closed arcs $\alpha,\beta$ in $\CA(\surfo)$,
their intersection number is an half integer in $\tfrac{1}{2}\ZZ$ and defined as follows
(following \cite{KS}):
\[
    \Int(\alpha,\beta)
    =\tfrac{1}{2}\Int_{\Tri}(\alpha,\beta)+\Int_{\surf-\Tri}(\alpha,\beta),
\]
where
\begin{gather}\label{eq:GIT}
    \Int_{\surfo-\Tri}(\alpha,\beta)=\min\{ |\alpha'\cap\beta'\cap \surfo-\Tri|
        \ff \alpha'\sim\alpha,\beta'\sim\beta \}
\end{gather}
and $$\Int_{\Tri}(\alpha,\beta)=\sum_{Z\in\Tri} |\{t\mid\alpha(t)=Z\}|\cdot|\{r\mid\beta(r)=Z\}|.$$
\end{itemize}
\subsection{Triangulations and flips (after Krammer)}\label{sec:Kra}

\begin{definition}
A \emph{triangulation} $\TT$ of $\surfo$ is a maximal collection of open arcs such that
\begin{itemize}
\item for any $\gamma_1,\gamma_2\in\TT$, $\Int(\gamma_1,\gamma_2)=0$;
\item $\TT$ is compatible with $\Tri$ in the sense that
the open arcs in $\TT$ divide $\surfo$ into $\aleph$ triangles,
each of which contains exactly one point in $\Tri$.
\end{itemize}
Let $\TT$ be a triangulation of $\surfo$ (consisting of $n$ open arcs).
The \emph{dual triangulation} $\TT^*$ of $\TT$ is the collection of $n$ closed arcs
in $\cA(\surfo)$, such that every closed arc only intersects one open arc in $\TT$
and with intersection one.
See the left picture of Figure~\ref{fig:ex0} for an example.
More precisely, for $\gamma$ in $\TT$, the corresponding closed arc in $\TT^*$
is the unique open arc $s$ that is contained in
the quadrilateral $A$ with diagonal $\gamma$,
connecting the two decorating points in $A$ and intersecting $\gamma$ only once.
We will call $s$ and $\gamma$ the dual of each other, w.r.t. $\TT$ (or $\TT^*$),
cf. left picture of Figure~\ref{fig:ex0}.
\end{definition}

There is a canonical map, the \emph{forgetful map} $$F\colon\surfo\to\surf,$$
forgetting the decorating points.
Clearly, $F$ induces a map from the set of open arcs in $\surfo$
to the set of open arcs in $\surf$.
And the image of a triangulation $\TT$ is still a triangulation $\T=F(\TT)$.
The (FST) quiver $Q_\TT$ associated to $\TT$ is defined to be
the (FST) quiver $Q_\T$ associated to $\T=F(\TT)$.
We proceed to introduce the notion of forward/backward flip of triangulations
(after \cite{Kr} and cf. \cite{KQ1}).

\begin{figure}[ht]\centering
\begin{tikzpicture}[scale=.4]
    \path (-135:4) coordinate (v1)
          (-45:4) coordinate (v2)
          (45:4) coordinate (v3);
\draw[NavyBlue,very thick](v1)to(v2)node{$\bullet$}to(v3);
    \path (-135:4) coordinate (v1)
          (45:4) coordinate (v2)
          (135:4) coordinate (v3);
\draw[NavyBlue,very thick](v2)node{$\bullet$}to(v3)node{$\bullet$}to(v1)node{$\bullet$}
(45:1)node[above]{$\gamma$};
\draw[>=stealth,NavyBlue,thick](-135:4)to(45:4);
\draw[red,thick](135:1.333)node{\tiny{$\circ$}}(-45:1.333)node{\tiny{$\circ$}};
\end{tikzpicture}
\begin{tikzpicture}[scale=1.2, rotate=180]
\draw[blue,<-,>=stealth](3-.6,.7)to(3+.6,.7);
\draw(3,.7)node[below,black]{\footnotesize{in $\surfo$}};
\draw[blue](3-.25,.5-.5)rectangle(3+.25,.5);\draw(3,1.5)node{};
\draw[blue,->,>=stealth](3-.25,.5-.5)to(3+.1,.5-.5);
\draw[blue,->,>=stealth](3+.25,.5)to(3-.1,.5);
\end{tikzpicture}
\begin{tikzpicture}[scale=.4];
    \path (-135:4) coordinate (v1)
          (-45:4) coordinate (v2)
          (45:4) coordinate (v3);
\draw[NavyBlue,very thick](v1)to(v2)node{$\bullet$}to(v3)
(45:1)node[above right]{$\gamma^\sharp$};
    \path (-135:4) coordinate (v1)
          (45:4) coordinate (v2)
          (135:4) coordinate (v3);
\draw[NavyBlue,very thick](v2)node{$\bullet$}to(v3)node{$\bullet$}to(v1)node{$\bullet$};
\draw[>=stealth,NavyBlue,thick](135:4).. controls +(-10:2) and +(45:3) ..(0,0)
                             .. controls +(-135:3) and +(170:2) ..(-45:4);
\draw[red,thick](135:1.333)node{\tiny{$\circ$}}(-45:1.333)node{\tiny{$\circ$}};
\end{tikzpicture}

\begin{tikzpicture}[scale=.4]
\draw[thick,>=stealth,->](0,5)to(0,3.6);\draw(0,4.3)node[left]{$^F$};
    \path (-135:4) coordinate (v1)
          (-45:4) coordinate (v2)
          (45:4) coordinate (v3);
\draw[NavyBlue,very thick](v1)to(v2)node{$\bullet$}to(v3);
    \path (-135:4) coordinate (v1)
          (45:4) coordinate (v2)
          (135:4) coordinate (v3);
\draw[NavyBlue,very thick](v2)node{$\bullet$}to(v3)node{$\bullet$}to(v1)node{$\bullet$};
\draw[>=stealth,NavyBlue,thick](-135:4)to(45:4);
\end{tikzpicture}
\begin{tikzpicture}[scale=1.2, rotate=180]
\draw(3,1.5)node{}(3,.5)node[above]{\footnotesize{in $\surf$}};
\draw(3-.6,.5)to(3+.6,.5);;
\end{tikzpicture}
\begin{tikzpicture}[scale=.4]
\draw[thick,>=stealth,->](0,5)to(0,3.6);\draw(0,4.3)node[right]{$^F$};
    \path (-135:4) coordinate (v1)
          (-45:4) coordinate (v2)
          (45:4) coordinate (v3);
\draw[NavyBlue,very thick](v1)to(v2)node{$\bullet$}to(v3);
    \path (-135:4) coordinate (v1)
          (45:4) coordinate (v2)
          (135:4) coordinate (v3);
\draw[NavyBlue,very thick](v2)node{$\bullet$}to(v3)node{$\bullet$}to(v1)node{$\bullet$};
\draw[>=stealth,NavyBlue,thick](135:4)to(-45:4);
\end{tikzpicture}
\caption{The flip}
\label{fig:flip}
\end{figure}
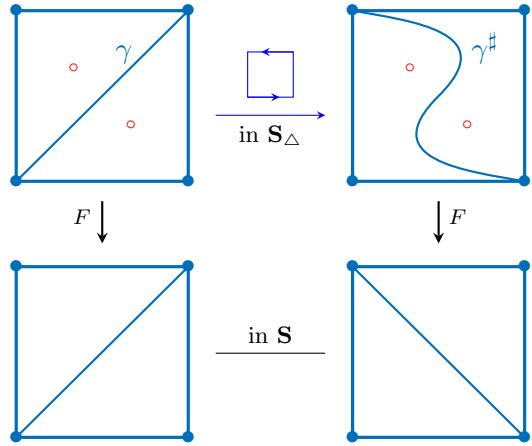

\begin{definition}
Let $\gamma$ be an open arc in a triangulation $\TT$ of $\surfo$.
The arc $\gamma^\sharp=\gamma^\sharp(\TT)$ is the arc obtained from $\gamma$
by anticlockwise moving its endpoints
along the quadrilateral in $\TT$ whose diagonal is $\gamma$
(cf. upper pictures of Figure~\ref{fig:flip}),
to the next marked points.
The \emph{forward flip} of a triangulation $\TT$ of $\surfo$ at $\gamma\in\TT$
is the triangulation $\tilt{\TT}{\sharp}{\gamma}$ obtained from $\TT$
by replacing the arc $\gamma$ with $\gamma^\sharp$.
Similarly, we can define arc $\gamma^\flat=\gamma^\flat(\TT)$ to be the arc obtained from $\gamma$
by clockwise moving its endpoints,
and the \emph{backward flip} $\tilt{\TT}{\flat}{\gamma}$ of $\TT$ at $\gamma\in\TT$
is the triangulation $\tilt{\TT}{\flat}{\gamma}$ obtained from $\TT$
by replacing the arc $\gamma$ with $\gamma^\flat$.
\end{definition}

Clearly, these two flips are inverse operations.
Also note that under the forgetful map $F$, a forward/backward flip in $\surfo$
becomes a normal flip (without direction) of $\surf$, cf. Figure~\ref{fig:flip},
which is an involution.
\begin{definition}
The exchange graph $\EG(\surfo)$ is the oriented graph whose vertices
are triangulations of $\surfo$ and whose edges correspond to forward flips between them.
\end{definition}

\begin{remark}
Recall that $\pi_1\EG(\surf)$ is generated by squares and pentagons (\cite[Theorem~3.10]{FST}).
By \cite{Kr}, forward flips also satisfy the square and pentagon relations
(cf. Figure~\ref{fig:5+}).
We believe that $\pi_1\EG(\surfo)$ is also generated by squares and pentagons.
\end{remark}
\begin{figure}
\begin{tikzpicture}[xscale=.6,yscale=.75]
  \foreach \j in {0, 8, 16}
  {
    \draw[thick](0+\j,0)node[NavyBlue]{$\bullet$}to(2+\j,0)node[NavyBlue]{$\bullet$}to(3+\j,2)node[NavyBlue]{$\bullet$}
    to(1+\j,3)node[NavyBlue]{$\bullet$}to(-1+\j,2)node[NavyBlue]{$\bullet$}to(0+\j,0);
    \draw(1+\j,1)node[red]{$\circ$}(\j,1.666)node[red]{$\circ$}(2+\j,1.666)node[red]{$\circ$};
  }
  \foreach \j in {4, 12}
  {
    \draw[thick](0+\j,0-6)node[NavyBlue]{$\bullet$}to(2+\j,0-6)node[NavyBlue]{$\bullet$}to(3+\j,2-6)node[NavyBlue]{$\bullet$}
    to(1+\j,3-6)node[NavyBlue]{$\bullet$}to(-1+\j,2-6)node[NavyBlue]{$\bullet$}to(0+\j,0-6);
    \draw(1+\j,1-6)node[red]{$\circ$}(\j,1.666-6)node[red]{$\circ$}(2+\j,1.666-6)node[red]{$\circ$};
  }
\draw[NavyBlue,thick](0,0)to(1,3)to(2,0) (0,4.5)node{};
\draw[NavyBlue,thick](-1+8,2).. controls +(5:3) and +(170:3) ..(2+8,0)to(1+8,3);
\draw[NavyBlue,thick](2+8+8,0).. controls +(170:3) and +(5:3) ..(-1+8+8,2).. controls +(15:3.5) and +(-135:2.5) ..(3+8+8,2);
\draw[NavyBlue,thick](1+4,3-6)to(0+4,0-6).. controls +(65:3.6) and +(-120:3) ..(3+4,2-6);
\draw[NavyBlue,thick](0+12,0-6).. controls +(65:3.6) and +(-120:3) ..(3+12,2-6)
    .. controls +(-135:2.5) and +(15:3.5) ..(-1+12,2-6);

\draw[blue,->,>=stealth] (1+4-1.5,1.5)to(1+4+1.5,1.5);
\draw[blue,->,>=stealth] (1+12-1.5,1.5)to(1+12+1.5,1.5);
\draw[blue,->,>=stealth] (1+2-1,1-3+1.5)to(1+2+.5,1-3-1);
\draw[blue,->,>=stealth] (1+6+8-.5,1-3-1)to(1+6+8+1,1-3+1.5);
\draw[blue,->,>=stealth] (1+8-1.5,1.5-6)to(1+8+1.5,1.5-6);

    \path (1+4,1.5+1) coordinate (v);

\draw[xshift=4.5cm,yshift=2cm,blue,>=stealth]
    (0,0)to(2*.4,0)to(3*.4,2*.4)to(1*.4,3*.4)to(-1*.4,2*.4)to(0,0)
    (.8,0)to(.4,1.2)
    (0,0)edge[->-=.7](.8,0) (.4,1.2)edge[->-=.7](-.4,.8);
\draw[xshift=12.5cm,yshift=2cm,blue,>=stealth]
    (0,0)to(2*.4,0)to(3*.4,2*.4)to(1*.4,3*.4)to(-1*.4,2*.4)to(0,0)
    (.8,0)to(-.4,.8)
    (1.2,.8)edge[-<-=.7](.8,0) (.4,1.2)edge[->-=.7](-.4,.8);
\draw[xshift=1cm,yshift=-2.5cm,blue,>=stealth]
    (0,0)to(2*.4,0)to(3*.4,2*.4)to(1*.4,3*.4)to(-1*.4,2*.4)to(0,0)
    (1.2,.8)edge[-<-=.7](.8,0) (.4,1.2)edge[->-=.7](0,0);
\draw[xshift=16.2cm,yshift=-2.5cm,blue,>=stealth]
    (0,0)to(2*.4,0)to(3*.4,2*.4)to(1*.4,3*.4)to(-1*.4,2*.4)to(0,0)
    (0,0)edge[->-=.7](.8,0) (-.4,.8)edge[-<-=.5](1.2,.8);
\draw[xshift=8.5cm,yshift=-6cm,blue,>=stealth]
    (0,0)to(2*.4,0)to(3*.4,2*.4)to(1*.4,3*.4)to(-1*.4,2*.4)to(0,0)
    (1.2,.8)edge[-<-=.7](0,0) (.4,1.2)edge[->-=.7](-.4,.8);
\end{tikzpicture}
\caption{The pentagon relation for forward flips}
\label{fig:5+}
\end{figure}
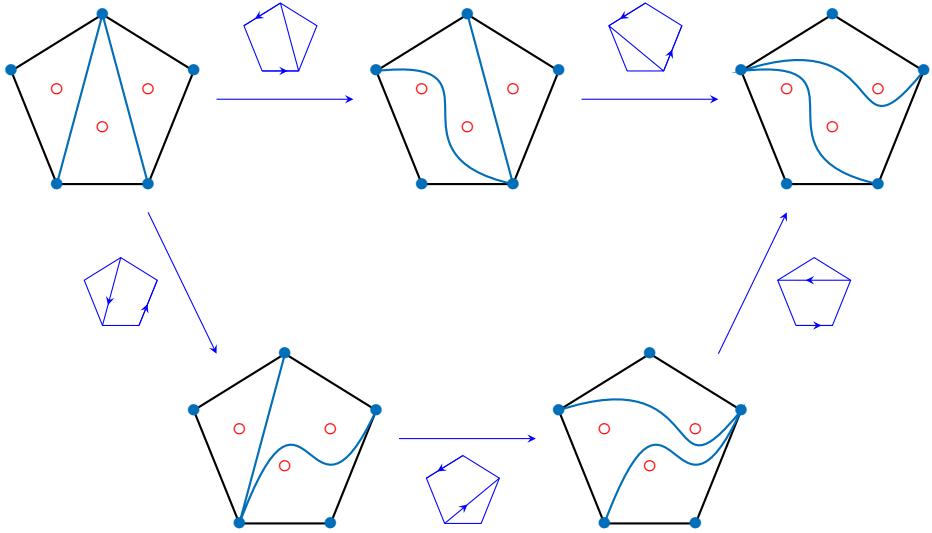
\subsection{The braid twists}
The mapping class group $\MCG(\surfo)$ is the group of isotopy classes of
(orientation preserving) homeomorphisms of $\surfo$,
where all homeomorphisms and isotopies are required to:
i) fix $\partial\surfo(\supset\M)$ pointwise;
ii) fix the decorating points set $\Tri$ (but allow to permutate points in it).
Note that the mapping class group $\MCG(\surf)$ of $\surf$ will require only the first condition
and thus there is a canonical map
\begin{gather}\label{eq:surjection}
    F_*\colon\MCG(\surfo)\twoheadrightarrow\MCG(\surf)
\end{gather}
induced by the forgetful map $F$.
As $\MCG(\surf)$ is generated by 
Dehn twists along simple closed curves (which misses
the decorating points), $F_*$ is clearly surjective.

For any closed arc $\eta\in\cA(\surfo)$, there is the (positive) \emph{braid twist}
$\Bt{\eta}\in\MCG(\surfo)$ along $\eta$, which is shown in Figure~\ref{fig:Braid twist}.
\begin{figure}[ht]\centering
\begin{tikzpicture}[scale=.3]
  \draw[very thick,NavyBlue](0,0)circle(6)node[above,black]{$_\eta$};
  \draw(-120:5)node{+};
  \draw(-2,0)edge[red, very thick](2,0)  edge[cyan,very thick, dashed](-6,0);
  \draw(2,0)edge[cyan,very thick,dashed](6,0);
  \draw(-2,0)node[white] {$\bullet$} node[red] {$\circ$};
  \draw(2,0)node[white] {$\bullet$} node[red] {$\circ$};
  \draw(0:7.5)edge[very thick,->,>=latex](0:11);\draw(0:9)node[above]{$\Bt{\eta}$};
\end{tikzpicture}\;
\begin{tikzpicture}[scale=.3]
  \draw[very thick, NavyBlue](0,0)circle(6)node[above,black]{$_\eta$};
  \draw[red, very thick](-2,0)to(2,0);
  \draw[cyan,very thick, dashed](2,0).. controls +(0:2) and +(0:2) ..(0,-2.5)
    .. controls +(180:1.5) and +(0:1.5) ..(-6,0);
  \draw[cyan,very thick,dashed](-2,0).. controls +(180:2) and +(180:2) ..(0,2.5)
    .. controls +(0:1.5) and +(180:1.5) ..(6,0);
  \draw(-2,0)node[white] {$\bullet$} node[red] {$\circ$};
  \draw(2,0)node[white] {$\bullet$} node[red] {$\circ$};
\end{tikzpicture}
\caption{The Braid twist}
\label{fig:Braid twist}
\end{figure}
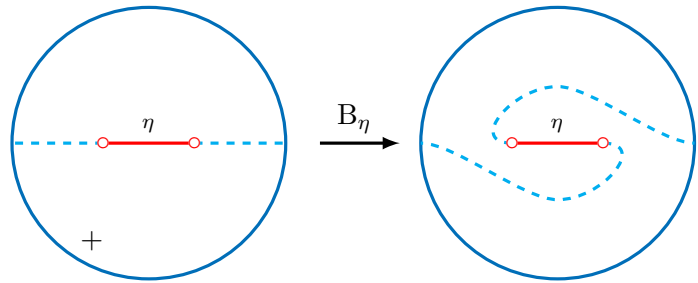
Further, there is the following well-known formula
\begin{gather}\label{eq:formulaB}
    \Bt{\Psi(\eta)}=\Psi\circ\Bt{\eta}\circ\Psi^{-1},
\end{gather}
for any $\Psi\in\MCG(\surfo)$.

\begin{definition}\label{def:bt}
The \emph{braid twist group} $\BT(\surfo)$ of the decorated marked surface $\surfo$
is the subgroup of $\MCG(\surfo)$ generated by the braid twists $\Bt{\eta}$
for $\eta\in\cA(\surfo)$.
\end{definition}

\begin{example}\label{ex:1/2}
If $\Int(\alpha,\beta)=\frac{1}{2}$, there is a closed arc $\eta$ (cf. Figure~\ref{fig:eta}) such that
\begin{gather}\label{eq:tt}
    \eta=\Bt{\alpha}(\beta)=\bt{\beta}(\alpha),\quad
    \alpha=\Bt{\beta}(\eta)=\bt{\eta}(\beta),\quad
    \beta=\Bt{\eta}(\alpha)=\bt{\alpha}(\eta).
\end{gather}
Note that $\eta$ is the closed arc such that
the interior of the triangle bounded by $\alpha,\beta,\eta$ is contractible.
In fact, there is exactly one more such closed arc
(dashed arc in Figure~\ref{fig:eta}), namely
\[\eta'=\bt{\alpha}(\beta)=\Bt{\beta}(\alpha),\]
satisfying the triangle bounded by these three arcs is contractible.

\begin{figure}[ht]\centering
\begin{tikzpicture}[scale=1]
\draw[red,thick](0,0).. controls +(90:1) and +(-90:1) ..(3,2);
\draw[red,thick](6,0).. controls +(180:1) and +(0:1) ..(3,2);

\draw[red,thick,dashed](0,0).. controls +(90:4) and +(90:4) ..(6,0);

\draw[red,thick](0,0).. controls +(90:.7) and +(210:1) ..(3,1)
    .. controls +(30:1) and +(180:1) ..(6,0);

\draw(0,0)node[white] {$\bullet$} node[red]{$\circ$}node[left]{$Z_2$}
     (3,2)node[white] {$\bullet$} node[red]{$\circ$}node[above]{$Z_0$}
     (3,3)node[above,black]{$\eta'$}
     (6,0)node[white] {$\bullet$} node[red]{$\circ$}node[right]{$Z_1$}
     (4.8,1)node {$\alpha$}(1.5,1.3)node {$\beta$}(3,0.7)node[below] {$\eta$};
\end{tikzpicture}
\caption{Intersecting one half}
\label{fig:eta}
\end{figure}
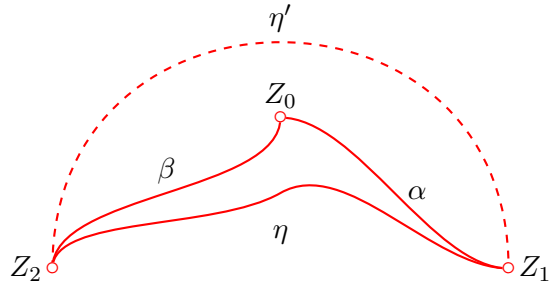
\end{example}

We have the following straightforward observation:
\begin{lemma}
Let $\gamma$ be a open arc in $\TT$ and $s$ be its dual in $\TT^*$.
Then in the triangulation $\tilt{\TT}{\sharp}{\gamma}$, the dual of $\gamma^\sharp$ is still $s$.
Moreover, let $\tilt{\TT}{\sharp}{\gamma}$ and $\tilt{\TT}{\flat}{\gamma}$ be
the two flips of $\TT$ at $\gamma$. Then
$$\gamma^\flat=\Bt{s}(\gamma^\sharp)\quad\text{and}\quad
\tilt{\TT}{\flat}{\gamma}=\Bt{s}(\tilt{\TT}{\sharp}{\gamma}).$$
\end{lemma}
\begin{proof}
The first claim follows from the upper pictures in Figure~\ref{fig:flip}
and the equations follow from Figure~\ref{fig:flips}.
\end{proof}

\begin{figure}[ht]\centering
\begin{tikzpicture}[xscale=-.4,yscale=.4,rotate=90]
    \path (-135:4) coordinate (v1)
          (-45:4) coordinate (v2)
          (45:4) coordinate (v3);
\draw[NavyBlue,very thick](v1)to(v2)node{$\bullet$}to(v3);
    \path (-135:4) coordinate (v1)
          (45:4) coordinate (v2)
          (135:4) coordinate (v3);
\draw[NavyBlue,very thick](v2)node{$\bullet$}to(v3)node{$\bullet$}to(v1)node{$\bullet$};
\draw[>=stealth,NavyBlue,thick](135:4).. controls +(-10:2) and +(45:3) ..(0,0)
                             .. controls +(-135:3) and +(170:2) ..(-45:4);
\draw[red,thick](135:1.333)node{\tiny{$\circ$}}(-45:1.333)node{\tiny{$\circ$}};
\end{tikzpicture}
\begin{tikzpicture}[scale=1.2, rotate=180]
\draw[blue,<-,>=stealth](3-.6,2.2)to(3+.6,2.2);
\draw[blue](3-.25,1.5+.5)rectangle(3+.25,1.5);\draw(3,3)node{};
\draw[blue,->,>=stealth](3-.25,1.5+.5)to(3-.25,1.5+.15);
\draw[blue,->,>=stealth](3+.25,1.5)to(3+.25,1.5+.35);
\end{tikzpicture}
\begin{tikzpicture}[scale=.4]
    \path (-135:4) coordinate (v1)
          (-45:4) coordinate (v2)
          (45:4) coordinate (v3);
\draw[NavyBlue,very thick](v1)to(v2)node{$\bullet$}to(v3);
    \path (-135:4) coordinate (v1)
          (45:4) coordinate (v2)
          (135:4) coordinate (v3);
\draw[NavyBlue,very thick](v2)node{$\bullet$}to(v3)node{$\bullet$}to(v1)node{$\bullet$};
\draw[>=stealth,NavyBlue,thick](-135:4)to(45:4);
\draw[red,thick](135:1.333)node{\tiny{$\circ$}}(-45:1.333)node{\tiny{$\circ$}};
\end{tikzpicture}
\begin{tikzpicture}[scale=1.2, rotate=180]
\draw[blue,<-,>=stealth](3-.6,.7)to(3+.6,.7);
\draw[blue](3-.25,.5-.5)rectangle(3+.25,.5);\draw(3,1.5)node{};
\draw[blue,->,>=stealth](3-.25,.5-.5)to(3+.1,.5-.5);
\draw[blue,->,>=stealth](3+.25,.5)to(3-.1,.5);
\end{tikzpicture}
\begin{tikzpicture}[scale=.4]
    \path (-135:4) coordinate (v1)
          (-45:4) coordinate (v2)
          (45:4) coordinate (v3);
\draw[NavyBlue,very thick](v1)to(v2)node{$\bullet$}to(v3);
    \path (-135:4) coordinate (v1)
          (45:4) coordinate (v2)
          (135:4) coordinate (v3);
\draw[NavyBlue,very thick](v2)node{$\bullet$}to(v3)node{$\bullet$}to(v1)node{$\bullet$};
\draw[>=stealth,NavyBlue,thick](135:4).. controls +(-10:2) and +(45:3) ..(0,0)
                             .. controls +(-135:3) and +(170:2) ..(-45:4);
\draw[red,thick](135:1.333)node{\tiny{$\circ$}}(-45:1.333)node{\tiny{$\circ$}};
\end{tikzpicture}
\caption{Composition of forward flips as a negative braid twist}
\label{fig:flips}
\end{figure}
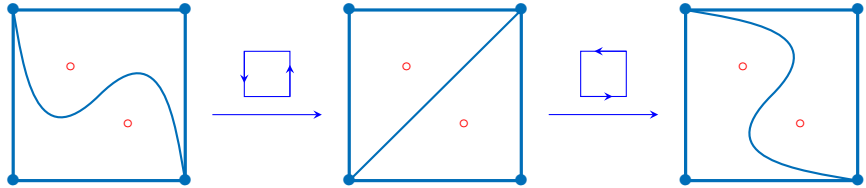

As a consequence, we obtain a map between exchange graphs.

\begin{lemma}
As graphs, we have the following surjective map induced by the forgetful map $F$:
\begin{gather}
    F_*\colon\EG(\surfo)/\BT(\surfo)\twoheadrightarrow\EG(\surf).
\end{gather}
\end{lemma}
\begin{proof}
Recall that there is a canonical surjection $F_*\colon\MCG(\surfo)\twoheadrightarrow\MCG(\surf)$
in \eqref{eq:surjection}.
By definition, it is straightforward to see that
\begin{gather}\label{eq:F=1}
    \BT(\surfo)\subset\ker F_*.
\end{gather}
Thus, $F$ induces a quotient map $F_*\colon\EG(\surfo)/\BT(\surfo)\to\EG(\surf)$
between sets.
Next, the $F_*$ preserves the edges (cf. Figure~\ref{fig:flip}),
in the sense that the forward and backward flips of a triangulation $\TT$ at some
closed $\gamma$ both become the flip of $\T=F(\TT)$ at $F(\gamma)$.
Thus, $F_*$ is a map between graphs.
Finally, by definition, $\EG(\surfo)$ is an oriented $(n,n)$-regular graph
(that is, every vertex has $n$ arrows in and $n$ arrow out) and
$\EG(\surf)$ is an unoriented $n$-regular graph.
Therefore we deduce that $F$ is surjective.
\end{proof}

\begin{remark}\label{rem:c.c.}
In fact, if we take any connected component $\EG^{\chi}(\surfo)$ of $\EG(\surfo)$,
then $F_*$ induces an isomorphism
\[
    F_*\colon\EG^{\chi}(\surfo)/\BT(\surfo)\cong\EG(\surf)
\]
since $\EG(\surf)$ is connected and both graphs are $n$-regular.
\end{remark}

\subsection{The initial triangulation}
\begin{remark}\label{rem:sp}
For technical reasons, we will exclude two special cases for the moment:
\begin{itemize}
\item[I).] an annulus with one marked point on each of its boundary components;
\item[II).] a torus with only boundary component and one marked point.
\end{itemize}
These cases will be discussed independently in \S~\ref{sec:Kro}.
\end{remark}

\begin{lemma}\label{lem:ini}
There exists a triangulation $\TT$ of $\surfo$ such that
any two triangles share at most one edge.
In other words, the quiver $Q_\TT$ has no double arrows.
\end{lemma}
\begin{proof}
The second statement, which is equivalent to the first one,
follows from \cite[Proposition~7.13]{GLFS}, noticing that we have excluded
the two special cases (a torus with one marked point and an annulus with two marked points).
\end{proof}

\begin{notations}
We will fix a triangulation $\TT_0$ such that its image $\T_0=F(\TT_0)$ (a triangulation of $\surf$)
satisfies the condition in Lemma~\ref{lem:ini}.
Let
\[\TT_0=\{\gamma_1,\ldots,\gamma_n\}\quad\text{and}\quad\TT_0^*=\{s_1,\ldots,s_n\},\]
where $s_i$ is the dual of $\gamma_i$ w.r.t. $\TT_0$.
Denote by $\EGp(\surfo)$ the connected component of $\EG(\surfo)$
that contains $\TT_0$.
\end{notations}

We say a curve is in a \emph{minimal position} w.r.t. $\TT_0$,
if it has minimal intersections with (arcs in) $\TT_0$.
Let $\Int(\TT_0,\eta)=\sum_{i=1}^n \Int(\gamma_i,\eta)$.
Then a representative $\eta$ is in a minimal position if
it intersects $\TT_0$ exactly $\Int(\TT_0,\eta)$ times.

We will repeatedly use induction on $\Int(\TT_0,\eta)$ later.
The next lemma is the basic idea of those inductions.

\begin{lemma}\label{lem:dcp}
Suppose that a general closed arc $\eta$ in $\CA(\surfo)$ is not a closed arc $s$ in $\TT_0^*$.
Then there are two closed arcs $\alpha,\beta$ in $\cA(\surfo)$
such that
\numbers
\item$\Int(\TT_0,\eta)=\Int(\TT_0,\alpha)+\Int(\TT_0,\beta)$ and
\item $\alpha,\beta,\eta$ form a contractible triangle in $\surfo$.
\ends
In the case when $\eta\in\cA(\surfo)$, $2^\circ$ is equivalent to
\begin{itemize}
\item[$\widetilde{2}^\circ$] $\Int(\alpha,\beta)=\frac{1}{2}$ and $\eta=\Bt{\alpha}(\beta)$.
\end{itemize}
\end{lemma}
\begin{proof}
Recall that we require that any two triangles in $\TT_0$
share at most one edge.
Thus if $\eta$ only intersects two triangles of $\TT_0$,
then $\eta=s_j\in\TT_0$ for some $j$ which we will exclude. 
Now suppose that $\eta$ intersects at least three triangles of $\TT_0$.
Then one of the decorating points in these triangles is not an endpoint of $\eta$.
Denote the triangle by $\Lambda_0$ with the decorating point $Z_0$ inside.
Choose a representative of $\eta$, also denoted by $\eta$,
when there is no confusion, such that it is in a minimal position
w.r.t. $\TT_0$.
One can draw a line segment $l$ from $Z_0$ to some point $Y$ of $\eta$
within $\Lambda_0$ such that $l$ doesn't intersect $\eta$ except
at the endpoints (cf. Figure~\ref{fig:line}).
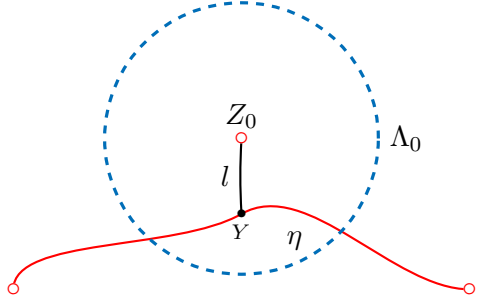
\begin{figure}[t]\centering
\begin{tikzpicture}[scale=1]
\draw[red,thick](0,0).. controls +(90:.7) and +(210:1) ..(3,1)
    node[below,black]{$^Y$}node[black]{\tiny{$\bullet$}}
    .. controls +(30:1) and +(180:1) ..(6,0);
\draw[bend right=-3,thick](3,1)to(3,2);
\draw[NavyBlue,dashed,very thick](3,2)circle(1.8cm)
    (4.8,2)node[right,black]{$\Lambda_0$}(3,1.5)node[left,black]{$l$};

\draw(0,0)node[white] {$\bullet$} node[red]{$\circ$}node[left]{}
     (3,2)node[white] {$\bullet$} node[red]{$\circ$}node[above]{$Z_0$}
     (6,0)node[white] {$\bullet$} node[red]{$\circ$}node[right]{}
     (3.7,0.9)node[below] {$\eta$};
\end{tikzpicture}
\caption{The line segment $l$}
\label{fig:line}
\end{figure}
Let $Z_1$ and $Z_2$ be the endpoints of $\eta$
such that $l$ is in the left side when we pass from $Z_2$ to $Z_1$.
Consider two closed arcs $\alpha$ and $\beta$ which are
isotopic to $l\cup\eta\mid_{Z_1 Y}$ and $l\cup\eta\mid_{Z_2 Y}$
respectively (cf. Figure~\ref{fig:alphabeta}).
Clearly, $2^\circ$ is satisfied.
Since $\eta$ is in a minimal position (w.r.t. $\TT_0$),
so are $\alpha$ and $\beta$. Thus $1^\circ$ is also satisfied.

Moreover, $\eta$ is one of $\Bt{\alpha}(\beta)$ and $\bt{\alpha}(\beta)$
when the endpoints of $\eta$ do not coincide (i.e. $\eta\in\cA(\surfo)$).
Thus, by choosing $\alpha$ and $\beta$ in some order we will obtain $\widetilde{2}^\circ$ as required.
\end{proof}

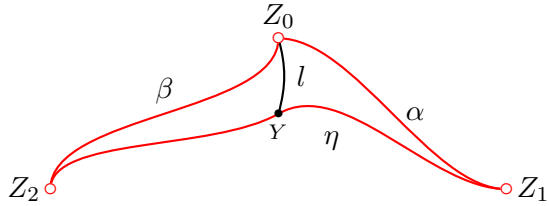
\begin{figure}[ht]\centering
\begin{tikzpicture}[scale=1]
\draw[red,thick](0,0).. controls +(90:1) and +(-90:1) ..(3,2);
\draw[red,thick](6,0).. controls +(180:1) and +(0:1) ..(3,2);

\draw[red,thick](0,0).. controls +(90:.7) and +(210:1) ..(3,1)
    node[below,black]{$^Y$}node[black]{\tiny{$\bullet$}}
    .. controls +(30:1) and +(180:1) ..(6,0);
\draw[bend right=15,thick](3,1)to(3,2) (3.1,1.5)node[right]{$l$};

\draw(0,0)node[white] {$\bullet$} node[red]{$\circ$}node[left]{$Z_2$}
     (3,2)node[white] {$\bullet$} node[red]{$\circ$}node[above]{$Z_0$}
     (6,0)node[white] {$\bullet$} node[red]{$\circ$}node[right]{$Z_1$}
     (4.8,1)node {$\alpha$}(1.5,1.3)node {$\beta$}(3.7,0.9)node[below] {$\eta$};
\end{tikzpicture}
\caption{Decomposing $\eta$}
\label{fig:alphabeta}
\end{figure}

\section{On the braid twist groups}\label{sec:4}
\subsection{Generators}
Recall that we have the braid twist group for $\surfo$.
Now we define the braid twist group for $\TT_0$.
\begin{definition}\label{def:BT}
Let $\TT$ be a triangulation of $\surfo$.
The \emph{braid twist group} $\BT(\TT)$ of the triangulation $\TT$
is the subgroup of $\MCG(\surfo)$ generated by the braid twists $\Bt{s}$
for the closed arcs $s$ in $\TT^*$.
\end{definition}

In fact, these two groups are the same.

\begin{lemma}\label{lem:gg}
$\BT(\surfo)=\BT(\TT_0)$.
\end{lemma}
\begin{proof}
Use induction on $\Int(\TT_0,\eta)$ to show that $\Bt{\eta}$ is in $\BT(\TT_0)$.
If so, then $\BT(\surfo)\subset\BT(\TT_0)$. Clearly, $\BT(\surfo)\supset\BT(\TT_0)$
and therefore the lemma follows.

If $\Int(\TT_0,\eta)=1$, then $\eta\in\TT_0^*$ and the claim is trivial.
Suppose that the claim holds for any $\eta'$ with $\Int(\TT_0,\eta')<m$.
Consider a closed arc $\eta\in\cA(\surfo)$ with $\Int(\TT_0,\eta)=m$.
Applying Lemma~\ref{lem:dcp}, we have $\eta=\Bt{\alpha}(\beta)$ for some $\alpha,\beta$.
By the inductive assumption, $\Bt{\alpha}$ and $\Bt{\beta}$ are in $\BT(\TT_0)$.
By formula \eqref{eq:formulaB}, we have
\[
    \Bt{\eta}=\Bt{\Bt{\alpha}(\beta)}=\Bt{\alpha}\circ\Bt{\beta}\circ\Bt{\alpha}^{-1}\in\BT(\TT_0),
\]
which completes the proof.
\end{proof}

\begin{proposition}\label{pp:gg}
$\BT(\surfo)=\BT(\TT)$ for any $\TT\in\EG(\surfo)$.
\end{proposition}
\begin{proof}
First, if $\TT_1$ and $\TT_2$ are related by a flip,
then their dual graphs are related by a Whitehead move, with respect to the corresponding
closed arc $\eta$ (which is unchanged under the flip), see Figure~\ref{fig:WH}.
Notice that the changes of closed arcs in $\TT_i^*$ are given by the braid twist
$\operatorname{B}_{\eta}^{\pm1}$.
Then by \eqref{eq:formulaB} it is straightforward to see that $\BT(\TT_1)=\BT(\TT_2)$.
By Lemma~\ref{lem:gg}, the proposition holds for any $\TT\in\EGp(\surfo)$.

As for $\TT$ in other connected components of $\EG(\surfo)$,
we can always find one triangulation in that component satisfying
the condition in Lemma~\ref{lem:ini}.
Then Lemma~\ref{lem:gg} would apply to that triangulation and thus
the proposition holds for any $\TT\in\EG(\surfo)$.
\end{proof}

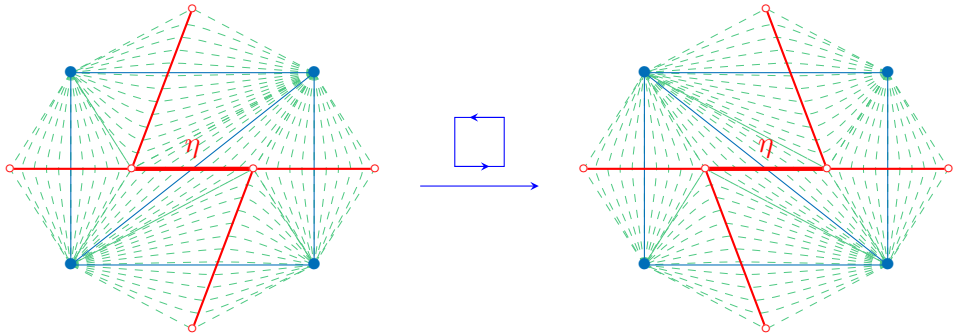
\begin{figure}[ht]\centering
\begin{tikzpicture}[xscale=-.4,yscale=.425]
    \path (4,3) coordinate (v2)
          (-4,3) coordinate (v1)
          (2,0) coordinate (v3)
          (0,5) coordinate (v4);
  \draw[blue!30!green!70, dashed,very thin] plot [smooth,tension=0] coordinates {(v1)(v3)(v2)};
  \draw[blue!30!green!70, dashed,very thin] plot [smooth,tension=0] coordinates {(v1)(v4)(v2)};
  \foreach \j in {.1, .18, .26, .34, .42, .5,.58, .66, .74, .82, .9}
    {
      \path (v3)--(v4) coordinate[pos=\j] (m0);
      \draw[blue!30!green!70, dashed,very thin] plot [smooth,tension=.3] coordinates {(v1)(m0)(v2)};
    }
    \path (4,-3) coordinate (v1)
          (-4,-3) coordinate (v2)
          (-2,0) coordinate (v3)
          (0,-5) coordinate (v4);
  \draw[blue!30!green!70, dashed,very thin] plot [smooth,tension=0] coordinates {(v1)(v3)(v2)};
  \draw[blue!30!green!70, dashed,very thin] plot [smooth,tension=0] coordinates {(v1)(v4)(v2)};
  \foreach \j in {.1, .18, .26, .34, .42, .5,.58, .66, .74, .82, .9}
    {
      \path (v3)--(v4) coordinate[pos=\j] (m0);
      \draw[blue!30!green!70, dashed,very thin] plot [smooth,tension=.3] coordinates {(v1)(m0)(v2)};
    }
    \path (4,-3) coordinate (v2)
          (-4,3) coordinate (v1)
          (2,0) coordinate (v3)
          (-2,0) coordinate (v4);
  \draw[blue!30!green!70, dashed,very thin] plot [smooth,tension=0] coordinates {(v1)(v3)(v2)};
  \draw[blue!30!green!70, dashed,very thin] plot [smooth,tension=0] coordinates {(v1)(v4)(v2)};
  \foreach \j in {.13,.26,.39,.87,.74,.61}
    {
      \path (v3)--(v4) coordinate[pos=\j] (m0);
      \draw[blue!30!green!70, dashed,very thin] plot [smooth,tension=.3] coordinates {(v1)(m0)(v2)};
    }
    \path (4,3) coordinate (v1)
          (4,-3) coordinate (v2)
          (2,0) coordinate (v3)
          (6,0) coordinate (v4);
  \draw[blue!30!green!70, dashed,very thin] plot [smooth,tension=0] coordinates {(v1)(v3)(v2)};
  \draw[blue!30!green!70, dashed,very thin] plot [smooth,tension=0] coordinates {(v1)(v4)(v2)};
  \foreach \j in {.1,.2,.3,.4,.5,.6,.7,.8,.9}
    {
      \path (v3)--(v4) coordinate[pos=\j] (m0);
      \draw[blue!30!green!70, dashed,very thin] plot [smooth,tension=.3] coordinates {(v1)(m0)(v2)};
    }
    \path (-4,-3) coordinate (v2)
          (-4,3) coordinate (v1)
          (-2,0) coordinate (v3)
          (-6,0) coordinate (v4);
  \draw[blue!30!green!70, dashed,very thin] plot [smooth,tension=0] coordinates {(v1)(v3)(v2)};
  \draw[blue!30!green!70, dashed,very thin] plot [smooth,tension=0] coordinates {(v1)(v4)(v2)};
  \foreach \j in {.1,.2,.3,.4,.5,.6,.7,.8,.9}
    {
      \path (v3)--(v4) coordinate[pos=\j] (m0);
      \draw[blue!30!green!70, dashed,very thin] plot [smooth,tension=.3] coordinates {(v1)(m0)(v2)};
    }
\draw[NavyBlue,thin]
  (4,-3)node{$\bullet$}to(4,3)node{$\bullet$}to(-4,3)node{$\bullet$}to(-4,-3)node{$\bullet$}to(4,-3)to(-4,3);
\draw[red,ultra thick](2,0)to(-2,0);
\draw[red,thick]
  (-6,0)node[white]{\tiny{$\bullet$}}node{\tiny{$\circ$}}to
  (-2,0)node[white]{\tiny{$\bullet$}}node{\tiny{$\circ$}}to
  (2,0)node[white]{\tiny{$\bullet$}}node{\tiny{$\circ$}}to
  (6,0)node[white]{\tiny{$\bullet$}}node{\tiny{$\circ$}}
  (0,5)node[white]{\tiny{$\bullet$}}node{\tiny{$\circ$}}to(2,0)
  (0,-5)node[white]{\tiny{$\bullet$}}node{\tiny{$\circ$}}to(-2,0)
  (0,0)node[above]{$\eta$};
\end{tikzpicture}\quad
\begin{tikzpicture}[scale=1.3, rotate=180]
\draw[blue,<-,>=stealth](3-.6,.7)to(3+.6,.7);
\draw[blue](3-.25,.5-.5)rectangle(3+.25,.5);\draw(3,2.2)node{};
\draw[blue,->,>=stealth](3-.25,.5-.5)to(3+.1,.5-.5);
\draw[blue,->,>=stealth](3+.25,.5)to(3-.1,.5);
\end{tikzpicture}\quad
\begin{tikzpicture}[xscale=.4,yscale=.425]
    \path (4,3) coordinate (v1)
          (-4,3) coordinate (v2)
          (2,0) coordinate (v3)
          (0,5) coordinate (v4);
  \draw[blue!30!green!70, dashed,very thin] plot [smooth,tension=0] coordinates {(v1)(v3)(v2)};
  \draw[blue!30!green!70, dashed,very thin] plot [smooth,tension=0] coordinates {(v1)(v4)(v2)};
  \foreach \j in {.1, .18, .26, .34, .42, .5,.58, .66, .74, .82, .9}
    {
      \path (v3)--(v4) coordinate[pos=\j] (m0);
      \draw[blue!30!green!70, dashed,very thin] plot [smooth,tension=.3] coordinates {(v1)(m0)(v2)};
    }
    \path (4,-3) coordinate (v1)
          (-4,-3) coordinate (v2)
          (-2,0) coordinate (v3)
          (0,-5) coordinate (v4);
  \draw[blue!30!green!70, dashed,very thin] plot [smooth,tension=0] coordinates {(v1)(v3)(v2)};
  \draw[blue!30!green!70, dashed,very thin] plot [smooth,tension=0] coordinates {(v1)(v4)(v2)};
  \foreach \j in {.1, .18, .26, .34, .42, .5,.58, .66, .74, .82, .9}
    {
      \path (v3)--(v4) coordinate[pos=\j] (m0);
      \draw[blue!30!green!70, dashed,very thin] plot [smooth,tension=.3] coordinates {(v1)(m0)(v2)};
    }
    \path (4,-3) coordinate (v2)
          (-4,3) coordinate (v1)
          (2,0) coordinate (v3)
          (-2,0) coordinate (v4);
  \draw[blue!30!green!70, dashed,very thin] plot [smooth,tension=0] coordinates {(v1)(v3)(v2)};
  \draw[blue!30!green!70, dashed,very thin] plot [smooth,tension=0] coordinates {(v1)(v4)(v2)};
  \foreach \j in {.13,.26,.39,.87,.74,.61}
    {
      \path (v3)--(v4) coordinate[pos=\j] (m0);
      \draw[blue!30!green!70, dashed,very thin] plot [smooth,tension=.3] coordinates {(v1)(m0)(v2)};
    }
    \path (4,3) coordinate (v1)
          (4,-3) coordinate (v2)
          (2,0) coordinate (v3)
          (6,0) coordinate (v4);
  \draw[blue!30!green!70, dashed,very thin] plot [smooth,tension=0] coordinates {(v1)(v3)(v2)};
  \draw[blue!30!green!70, dashed,very thin] plot [smooth,tension=0] coordinates {(v1)(v4)(v2)};
  \foreach \j in {.1,.2,.3,.4,.5,.6,.7,.8,.9}
    {
      \path (v3)--(v4) coordinate[pos=\j] (m0);
      \draw[blue!30!green!70, dashed,very thin] plot [smooth,tension=.3] coordinates {(v1)(m0)(v2)};
    }
    \path (-4,-3) coordinate (v1)
          (-4,3) coordinate (v2)
          (-2,0) coordinate (v3)
          (-6,0) coordinate (v4);
  \draw[blue!30!green!70, dashed,very thin] plot [smooth,tension=0] coordinates {(v1)(v3)(v2)};
  \draw[blue!30!green!70, dashed,very thin] plot [smooth,tension=0] coordinates {(v1)(v4)(v2)};
  \foreach \j in {.1,.2,.3,.4,.5,.6,.7,.8,.9}
    {
      \path (v3)--(v4) coordinate[pos=\j] (m0);
      \draw[blue!30!green!70, dashed,very thin] plot [smooth,tension=.3] coordinates {(v1)(m0)(v2)};
    }
\draw[NavyBlue,thin]
  (4,-3)node{$\bullet$}to(4,3)node{$\bullet$}to(-4,3)node{$\bullet$}to(-4,-3)node{$\bullet$}to(4,-3)to(-4,3);
\draw[red,ultra thick](2,0)to(-2,0);
\draw[red,thick]
  (-6,0)node[white]{\tiny{$\bullet$}}node{\tiny{$\circ$}}to
  (-2,0)node[white]{\tiny{$\bullet$}}node{\tiny{$\circ$}}to
  (2,0)node[white]{\tiny{$\bullet$}}node{\tiny{$\circ$}}to
  (6,0)node[white]{\tiny{$\bullet$}}node{\tiny{$\circ$}}
  (0,5)node[white]{\tiny{$\bullet$}}node{\tiny{$\circ$}}to(2,0)
  (0,-5)node[white]{\tiny{$\bullet$}}node{\tiny{$\circ$}}to(-2,0)
  (0,0)node[above]{$\eta$};
\end{tikzpicture}
\caption{The Whitehead move, as the flip of the dual triangulations (red)}
\label{fig:WH}
\end{figure}

Besides, the closed arcs are `reachable', in the following sense.
\begin{proposition}\label{pp:oa}
Let $\TT\in\EG(\surfo)$.
For any $\eta\in\cA(\surfo)$,
there exists $s\in\TT^*$ and $b\in\BT(\surfo)$ such that $\eta=b(s)$, i.e.
$$\cA(\surfo)=\BT(\surfo)\cdot\TT^*.$$
\end{proposition}
\begin{proof}
Consider the case when $\TT=\TT_0$ first.
Then this follows easily by induction on $\Int(\TT_0,\eta)$, using Lemma~\ref{lem:dcp}.
Second, by the Whitehead move (cf. Figure~\ref{fig:WH}),
if $\TT_1$ and $\TT_2$ are related by a flip, then
$$\BT(\surfo)\cdot\TT_1^*=\BT(\surfo)\cdot\TT_2^*.$$
Therefore the proposition holds for $\TT\in\EGp(\surfo)$.
Finally, as in the last paragraph of the proof of Proposition~\ref{pp:gg},
we deduce that the proposition holds for any $\TT\in\EG(\surfo)$.
\end{proof}

\subsection{Centers}
We recall the definition of the braid groups (a.k.a. Artin groups) of type $A$ and $\widetilde{A}$.
\begin{definition}
Suppose that $Q$ is a quiver or a diagram as in \eqref{eq:labeling}:
\begin{gather}\label{eq:labeling}
\begin{array}{llr}
    \xymatrix{A_{n}:} \quad &
    \xymatrix{
        1 \ar@{-}[r]& 2 \ar@{-}[r]& \cdots \ar@{-}[r]& n }         \\
   \xymatrix@R=0.5pc{\widetilde{A_{n}}:} \quad &
    \xymatrix{
        1 \ar@{-}[r]& 2 \ar@{-}[r]& \cdots& \ar@{-}[r]& n\text{-}1 \\
        &&0\ar@{-}[urr]\ar@{-}[ull]}                                                           .
\end{array}
\end{gather}
Denote by $\Br(Q)$ the braid group associated to $Q$,
generated by $\mathbf{b}=\{b_i\mid i\in Q_0\}$ subject to the relations
\begin{gather*}
\begin{cases}
    b_j b_i b_j=b_i b_j b_i,  &
    \text{there is exactly one arrow between $i$ and $j$ in $Q$},\\
    b_i b_j=b_j b_i,  &  \text{otherwise}.
\end{cases}
\end{gather*}
\end{definition}
Let $Z_0^{\BT}$ be the center of $\BT(\surfo)$ and $\BT_*(\surfo)=\BT(\surfo)/Z_0^{\BT}$.
\begin{itemize}
\item If $\surfo$ is a polygon, then $\BT(\surfo)\cong\Br({A_{n}})$
and $Z_0^{\BT}$ is the infinite cyclic group generated by $\Dehn{\partial\surfo}$.
\item If $\surfo$ is an annulus, then $\BT(\surfo)\cong\Br({\widetilde{A_{n}}})$
and $Z_0^{\BT}=1$ (\cite{KP}).
\end{itemize}
We will show that $Z_0^{\BT}=1$ holds for all other cases.
Denote the boundary components of $\surfo$ by $\partial_j ,1\leq j\leq |\partial\surf|$.

\begin{lemma}\label{lem:cut}
By cutting along the (initial) closed arcs in $\TT_0^*$,
$\surfo$ will be divided into $m$ annuli $\mathbf{A}_i, 1\leq i\leq m$,
such that $\partial_i$ is a boundary component of $\mathbf{A}_i$.
\end{lemma}
\begin{proof}
For each boundary segment $\gamma\subset\partial\surfo$ that is in a triangle $T$ in $\TT_0$
with decorating point $Z$, denote by $\gamma^*$ its dual, which is
the simple arc in $T$ (unique up to isotopy) connecting $Z$ and the midpoint of $\gamma$.
Call the union of $\TT_0^*$ and the arcs $\gamma^*$ as above
(for all segments $\gamma$ in $\partial\surfo$) the full dual of $\TT_0$.
Denote it by $\widehat{\TT_0^*}$, see Figure~\ref{fig:full dual} for example.

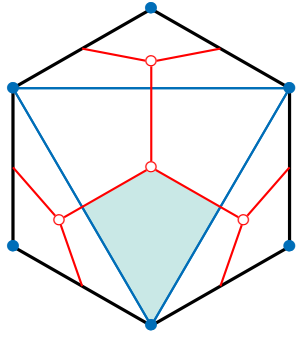
\begin{figure}[ht]\centering
\begin{tikzpicture}[scale=.35,rotate=60]
\draw[white, fill=Emerald!20](0,0)to(150:3)to(-150:6)to(-90:3)to(0,0);
\foreach \j in {1,...,6}{\draw[very thick](60*\j+30:6)to(60*\j-30:6);}
\foreach \j in {1,...,6}{\draw[NavyBlue,thick](120*\j-30:6)node{$\bullet$}to(120*\j+90:6)
    (120*\j-90:6)node{$\bullet$};}
\foreach \j in {1,...,3}{  \draw[thick,red](30+120*\j:4)to(0,0)
    (30+120*\j:4)to(60+120*\j:5.196)(30+120*\j:4)to(120*\j:5.196);}
\foreach \j in {1,...,3}{  \draw(30-120*\j:4)node[white]{$\bullet$}node[red]{$\circ$};}
\draw(0,0)node[white]{$\bullet$}node[red]{$\circ$};
\end{tikzpicture}
  \caption{The full dual of a triangulation}
  \label{fig:full dual}
\end{figure}

Then the surface $\surfo-\widehat{\TT_0^*}$ obtained from $\surfo$
by cutting along all arcs in $\widehat{\TT_0^*}$ satisfies the following:
\begin{itemize}
\item it consists of $|\M|$ connected components,
each of which contains exactly one marked point in $\M$;
\item each component is a disk, since it can be obtained
by gluing many quadrilaterals (cf. the shaded area in Figure~\ref{fig:full dual})
along some segment containing the marked point in that component.
\end{itemize}
Further, by gluing back along the arcs dual to boundary segments in $\surfo$,
we deduce that the surface $\surfo-{\TT_0^*}$ obtained from $\surfo$
by cutting along all arcs in ${\TT_0^*}$ satisfies the following:
\begin{itemize}
\item it consists of $|\partial|$ connected components;
\item each component $\mathbf{A}_i$ is an annulus,
such that one of the boundary components of $\mathbf{A}_i$
is a boundary component of $\surfo$.
\end{itemize}
Thus the lemma follows.
\end{proof}

\begin{proposition}\label{pp:trivial}
If $\surfo$ is neither a polygon nor an annulus, then $Z_0^{\BT}=1$.
\end{proposition}
\begin{proof}
Denote by $\Dehn{}(\partial\surfo)$ the subgroup of $\MCG(\surfo)$ generated by
the Dehn twist $\{\Dehn{\partial_i}\}$ of its boundary components.

We claim that
$Z_0^{\BT}\subset\Dehn{}(\partial\surfo).$
Let $z\in Z_0^{\BT}$. Then $z\circ \Bt{\eta}=\Bt{\eta}\circ z$ for any $\eta\in\cA(\surfo)$.
Hence by \eqref{eq:formulaB} we have
\begin{gather}\label{eq:commutes}
    \Bt{z(\eta)}=z\circ\Bt{\eta}\circ z^{-1}=\Bt{\eta}.
\end{gather}
Thus $z(\eta)=\eta$ for any $\eta\in\cA(\surfo)$,
which in particular implies that $z$ preserves $\Tri$ pointwise
(note that $|\Tri|=\aleph\geq3$ in our situation)
and $\TT_0^*$.
By Lemma~\ref{lem:cut},
cutting along closed arcs in $\TT_0^*$
divides $\surfo$ into $m$ annuli $\mathbf{A}_i$,
such that $\partial_i$ is a boundary component of $\mathbf{A}_i$.
Since $z$ preserves all such closed arcs,
it can be realized as composition of some element $z_i\in\MCG(\mathbf{A}_i)$ (where the order of the composition does not matter since they commute with each other).
Note that $\MCG(\mathbf{A}_i)$ is generated by $\Dehn{\partial_i}$,
which implies $z\in\Dehn{}(\partial\surfo)$.
Thus the claim holds.

There is also the subgroup $\Dehn{}(\partial\surf)$
of $\MCG(\surf)$ generated by the Dehn twist along its boundary components
and the obvious induced map $F_*\left(\Dehn{}(\partial\surfo)\right)=\Dehn{}(\partial\surf)$,
which sends $\Dehn{\partial_i}$ to $\Dehn{\partial_i}$.
Since $\surf$ is not a polygon, $\{\Dehn{\partial_i}\}$ are non-trivial
in both $\MCG(\surfo)$ and $\MCG(\surf)$.
Further, since $\surf$ is not an annulus,
$\{\Dehn{\partial_i}\}$ are distinct (and commute with each other).
Therefore, $F_*\colon \Dehn{}(\partial\surfo)\to\Dehn{}(\partial\surf)$
is an isomorphism.

Now combining $Z_0^{\BT}\subset\Dehn{}(\partial\surfo)$ and \eqref{eq:F=1},
we deduce that $F_*(Z_0^{\BT})=1$ in $\MCG(\surf)$ and hence $Z_0^{\BT}=1$ in $\MCG(\surfo)$.
\end{proof}

\section{From closed arcs to perfect objects}\label{sec:5}
\subsection{The Koszul dual and minimal model}
Let $\Gamma_\TT=\Gamma(Q_\TT,W_\TT)$ be the Ginzburg dg algebra obtained from a triangulation $\TT$.
Recall that there is a canonical heart $\h_\TT$ in $\D_{fd}(\Gamma_\TT)$
and let $$S_\TT=\bigoplus_{S\in\Sim\h_\TT} S$$
be the direct sum of the simples in $\h_\TT$.
Consider the (dg) endomorphism algebra
$
    \ee_\TT=\RHom(S_\TT, S_\TT).
$
By \cite{KN}, we have the following derived equivalence:
\begin{gather}\label{eq:DE}
\xymatrix@C=4pc{
    \D_{fd}(\Gamma_\TT) \ar@<.5ex>[rr]^{ \RHom_{\Gamma_\TT}(S_\TT, ?) }
        \ar@{<-}@<-.5ex>[rr]_{ ?\otimes^{\mathbf{L}}_{\ee_{\TT}}S_\TT }  && \per\ee_\TT,
}\end{gather}
In particular, $\{S\}_{S\in\h_\Gamma}$ in $\D_{fd}(\Gamma_\TT)$ become
(indecomposable) projectives in $\per\ee_{\TT}$.
By \cite[\S~A.15]{K8},
the multiplications in the A$_\infty$-structure of
the homology of $\ee_{\TT}$ are induced from differentials of $\Gamma_\TT$.
In particular, when $m\geq3$,
the $m$-multiplications are induced from
the $(m+1)$-cycle in the potential $W_\TT$,
which vanish in our case (since we only have 3-cycles in the potential).
This means that $\ee_{\TT}$ is formal and
hence is quasi-isomorphic to its homology (the minimal model), denoted by
\begin{gather}\label{eq:EE}
    \EE_\TT=\Hom^\bullet(S_\TT, S_\TT).
\end{gather}
which is just a graded algebra.
We will identify $\D_{fd}(\Gamma_\TT)$ with $\per\EE_\TT$ when there is no confusion.
Recall that the Ext quiver $\Q{\h}$ of a finite heart $\h$
is the (positively) graded quiver
whose vertices are the simples of $\h$ and
whose graded edges correspond to a basis of
$\End^{>0}(\bigoplus_{S\in\Sim\h} S)$.

\begin{example}\label{ex:2}
The Ext quiver of the canonical heart (in the corresponding 3-CY category)
of the quiver with potential in Example~\ref{ex:1} is shown as follows.
\begin{gather}\label{eq:ex2}
    \xymatrix@R=3pc@C=2.3pc{ &S_2  \ar@(ul,ur)[]^3
    \ar@<.8ex>[dr]^{1}\ar@<.2ex>[dl]^{2}
      \\ S_1\ar@(l,d)[]_3\ar@<.2ex>[rr]^{2}\ar@<.8ex>[ur]^{1}&&
        S_3\ar@(d,r)[]_3\ar@<.2ex>[ul]^{2}
          \ar@<.8ex>[ll]^{1}
  }
\end{gather}
Moreover, the differentials in \eqref{eq:ex} become the following multiplications:
\begin{equation}
\begin{array}{rll}
    \Hom^1(S_{i-1},S_{i})\otimes\Hom^1(S_i,S_{i+1})&\cong&
        \Hom^2(S_{i-1},S_{i+1}),\\
    \Hom^k(S_i, S_{i+k})\otimes\Hom^{3-k}(S_{i+k},S_i)&\cong&\Hom^3(S_i,S_i),
\end{array}\label{eq:rels}
\end{equation}
for $i=1,2,3$ and $k=0,1,2,3$, where $S_{i+3}=S_{i}$ for $i\in\ZZ$.
\end{example}

\begin{notations}
Recall that we have fixed an initial triangulation $\TT_0$.
\begin{itemize}
\item We will write $\Gamma_0$ for $\Gamma_{\TT_0}$ and
similar for $\EE_{\TT_0}$, $\h_{\TT_0}$ and so on.
\item Let the $\Gamma^i=e_i\Gamma_0$ be the indecomposable projective summands of $\Gamma_0$.
\item Let $S_1,\ldots,S_n$ be the simples in $\h_{0}$ which correspond to the projectives above.
Under the derived equivalence \eqref{eq:DE},
the $S_i$ become the (projective) summands of the silting objects $\EE_{0}$ in
$\per\EE_{0}\cong\D_{fd}(\Gamma_0)$.
\end{itemize}
\end{notations}

\subsection{The string model}\label{sec:sm}
\begin{definition}
A dg $\EE_0$-module $X$ is \emph{minimal perfect} if its underlying graded module
(denoted by $|X|$) is of the form
\begin{gather}\label{eq:graded}
    |X|=\bigoplus_{k=1}^l X_k,
\end{gather}
where each $X_k$ is a finite direct sum of shifted copies of direct summands of $\EE_0$
(i.e. copies of $S_j$) whose differential, as a degree $1$ map from $X$ to itself,
is a strictly upper triangular matrix whose entries are in the ideal of $\EE_0$ generated by the arrows in $\Q{\h_0}$. 
\end{definition}

\begin{figure}\centering
\begin{tikzpicture}[yscale=.6,xscale=.8,rotate=0]
\draw[NavyBlue, thick](-2,0)node{$\bullet$}to(2,0)node{$\bullet$}to(0,4)node{$\bullet$}to(-2,0)
    (0,1.5)node[red]{$\circ$};
\draw[red, thick,>=stealth](-1.6,-1) .. controls +(60:.1) and +(-120:.2) .. (-1.1,0);
\draw[red, thick,>=stealth](-1.1,0) .. controls +(60:.1) and +(180:.5) .. (0,.9);

\draw[red, thick,->,>=stealth](1.1,0) .. controls +(120:.1) and +(0:.5) .. (-.1,.9);
\draw[red, thick,>=stealth](1.6,-1) .. controls +(120:.1) and +(-60:.2) .. (1.1,0);

\draw(-1.1,0)node{$\bullet$}node[above]{}(1.1,0)node{$\bullet$}node[above]{};
\end{tikzpicture}
\caption{A digon intersected by some $\gamma$ and $\eta$}
\label{fig:bad}
\end{figure}
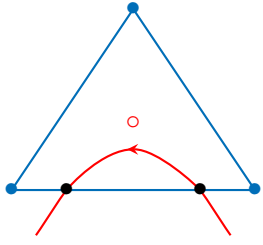
Let $\eta$ be a general closed arc in $\surfo$ such that
it is in a minimal position w.r.t. $\TT_0$.
This is equivalent to saying that there is no digon shown as in Figure~\ref{fig:bad}.
One can associate a minimal perfect dg $\EE_0$ module $X_\eta$ as follows
\begin{itemize}
\item its underlying graded module $|X_\eta|$ has the form as in \eqref{eq:graded}.
\item Let the endpoints of $\eta$ be $Z$ and $Z'$.
  Suppose that from $Z$ to $Z'$, $\eta$ intersects $\TT_0$ at
  $V_1,\ldots,V_m$ accordingly, where $V_i$ is in the arc $\gamma_{j_i}\in\TT_0$
  for $1\leq i\leq m$ and some $1\leq j_i\leq n$ (cf. Figure~\ref{fig:int0}).
  Note that since when choose $\eta$ in a minimal position w.r.t. $\TT_0$,
  $j_i$ are independent of the choice of $\eta$ (only depend the isotopy class of $\eta$).
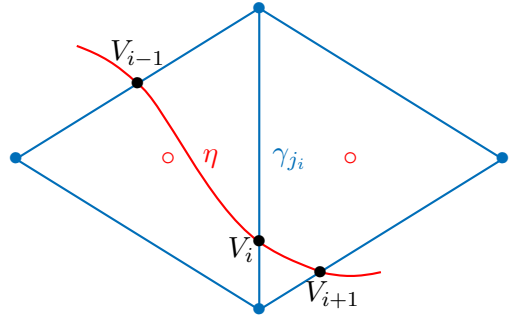
\begin{figure}[t]\centering
\begin{tikzpicture}[yscale=1,xscale=.8,rotate=90]
\draw[NavyBlue, thick](-2,0)node{$\bullet$}to(2,0)node{$\bullet$}to(0,4)node{$\bullet$}to(-2,0)
    to(0,-4)node{$\bullet$}to(2,0)
    (0,1.5)node[red]{$\circ$}(0,-1.5)node[red]{$\circ$};
\draw[](-1.8,-1.2)node[]{$V_{i+1}$} (1.4,2)node[]{$V_{i-1}$}
    (-1.2,.3)node[]{$V_{i}$};
\draw[red, thick,>=stealth](-1.5, -2)to[bend left=10](-1.5,-1)
    .. controls +(60:.1) and +(-120:.5) .. (-1.1,0);
\draw[red, thick,>=stealth](-1.1,0) .. controls +(60:1) and +(-120:.5) .. (1,2)
    to[bend left=-10](1.5,3);
\draw[red](0,.8)node{$\eta$};\draw[NavyBlue](0,-.5)node{$\gamma_{j_i}$};
\draw(-1.5,-1)node{$\bullet$}(-1.1,0)node{$\bullet$}(1,2)node{$\bullet$};
\end{tikzpicture}
\caption{The intersections between $\eta$ and $\TT_0$}
\label{fig:int0}
\end{figure}

\begin{figure}[b]\centering\;
\begin{tikzpicture}[yscale=.6,xscale=.8,rotate=0]
\draw[NavyBlue, thick](-2,0)node{$\bullet$}to(2,0)node{$\bullet$}to(0,4)node{$\bullet$}to(-2,0)
    (0,1.5)node[red]{$\circ$};
\draw[red, thick,>=stealth](-1.5,3.5)to[bend right](-.75,2.5)
    edge[bend right=10,->-=.5,>=stealth](.75,2.5);
\draw[red,thick](.75,2.5)to[bend right](1.5,3.5);

\draw(-.75,2.5)node{$\bullet$}node[above]{$^V$}(.75,2.5)node{$\bullet$}node[above]{$^{W}$};

\draw(0,-1-.5)node{$V \xrightarrow{\;a\;} W$}(0,-1.5-.5)node[below]{$\deg a=1$};

\draw[NavyBlue](0,-.3)node{$\gamma_3$};
\draw[NavyBlue](-1.85,1)node{$\gamma_1$}(1.8,1)node{$\gamma_2$};
\end{tikzpicture}
\qquad
\begin{tikzpicture}[yscale=.6,xscale=.8,rotate=0]
\draw[NavyBlue, thick](-2,0)node{$\bullet$}to(2,0)node{$\bullet$}to(0,4)node{$\bullet$}to(-2,0)
    (0,1.5)node[red]{$\circ$};

\draw[red, thick,>=stealth](-1,3)to[](-1,2);
\draw[red, thick,-<-=.5,>=stealth](-1,2) .. controls +(-90:1.5) and +(-90:1.5) .. (1,2);
\draw[red,thick](1,2)to[](1,3);

\draw(-1,2)node{$\bullet$}node[left]{$^V$}(1,2)node{$\bullet$}node[right]{$^{W}$};
\draw(0,-1-.5)node{$V \xleftarrow{\;a\;} W$}(0,-1.5-.5)node[below]{$\deg a=2$};
\draw[NavyBlue](0,-.3)node{$\gamma_3$};
\draw[NavyBlue](-1.85,1)node{$\gamma_1$}(1.8,1)node{$\gamma_2$};
\end{tikzpicture}
\qquad
\begin{tikzpicture}[yscale=.6,xscale=.8,rotate=0]
\draw[NavyBlue, thick](-2,0)node{$\bullet$}to(2,0)node{$\bullet$}to(0,4)node{$\bullet$}to(-2,0)
    (0,1.5)node[red]{$\circ$};
\draw[red, thick,>=stealth](-1.4,-1) .. controls +(60:.1) and +(-120:.2) .. (-1.1,0);
\draw[red, thick,->-=.75,>=stealth](-1.1,0) .. controls +(60:.1) and +(180:.5) .. (0,2);

\draw[red, thick,>=stealth](1.4,-1) .. controls +(120:.1) and +(-60:.2) .. (1.1,0);
\draw[red, thick,>=stealth](1.1,0) .. controls +(120:.1) and +(0:.5) .. (0,2);

\draw(-1.1,0)node{$\bullet$}node[above]{$^V$}
    (1.1,0)node{$\bullet$}node[above]{$^{W}$};
\draw(0,-1.5)node{$V \xrightarrow{\;a\;} W$}(0,-1.5-.5)node[below]{$\deg a=3$};
\draw[NavyBlue](0,-.3)node{$\gamma_3$};
\draw[NavyBlue](-1.85,1)node{$\gamma_1$}(1.8,1)node{$\gamma_2$};
\end{tikzpicture}
\caption{Inducing graded arrows}
\label{fig:arrows}
\end{figure}
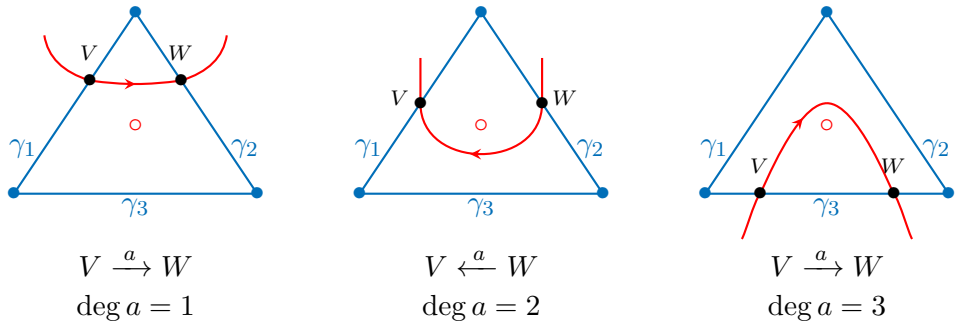
\item Each line segment $V_iV_{i+1}$ in $\eta$ induces a graded arrow $a_i$
  between $V_i$ and $V_{i+1}$
  (clockwise within the corresponding triangle).
  See Figure~\ref{fig:arrows} for how an edge $a$ in the Ext quiver $\hua{Q}(\h_0)$
  induces such a graded arrow $a$ between $V$ and $W$
  respectively.
  Then we obtain a string $H_\eta$, whose vertices are $V_i$'s and whose (graded) arrows
  are those induced arrows.
  \[
     H_\eta\colon\quad
     V_1 \frac{a_1}{\qquad} V_2 \frac{a_2}{\qquad}\cdots \frac{{}}{\qquad}
        V_{m-1} \frac{a_{m-1}}{\qquad} V_m
  \]
\item Each intersection $V_i$ corresponds to a summand $S_{j_i}[\delta_{i}]$ in some $X_{\delta_i}$
    for some integer $\delta_i$.
    So we have
    \[ \bigoplus_{k=1}^l X_k=\bigoplus_{i=1}^m S_{j_i}[\delta_i] .\]
\item For the arrow $a_i$, we have two cases:
  \numbers
  \item If its orientation is $V_{i}\to V_{i+1}$, then
    the degrees $\delta_i$ satisfy
    \begin{gather}\label{eq:1} \delta_{i+1}=\delta_i+\deg a_i-1.\end{gather}
    Moreover, the map $S_{j_i}\to S_{j_{i+1}}[\deg a_i]$ corresponding to $a_i$
    induces a degree $1$ map
    \[
        d_{a_i}\colon S_{j_i}[\delta_i] \xrightarrow{\quad1\quad} S_{j_{i+1}}[\delta_{i+1}].
    \]
  \item If its orientation is $V_{i}\leftarrow V_{i+1}$, then
    the degrees $\delta_i$ satisfy
    \begin{gather}\label{eq:2}\delta_{i}=\delta_{i+1}+\deg a_i-1.\end{gather}
    Moreover, the map $S_{j_{i+1}}\to S_{j_i}[\deg a_i]$ corresponding to $a_i$
    induces a degree $1$ map
    \[
        d_{a_i}\colon S_{j_{i+1}}[\delta_{i+1}] \xrightarrow{\quad1\quad} S_{j_{i}}[\delta_{i}].
    \]
  \ends
\item Finally, the differential $\diff_\eta$ of ${X_\eta}$ is given by the degree 1 map
    \[\diff_\eta=\sum_{i=1}^{m-1} d_{a_i}.\]
\end{itemize}

\begin{lemma}\label{lem:wd}
The complex $X_\eta$ above is well-defined.
\end{lemma}
\begin{proof}
We only need to check $\diff_\eta^2=0$, i.e.
$d_{a_{i+1}}d_{a_i}=0$ and $d_{a_{i}}d_{a_{i+1}}=0$
for any $i$ (when they make sense in $\diff_\eta^2$).

On one hand, since $\eta$ is in a minimal position w.r.t. $\TT_0$
and any two triangles in $\TT_0$ share at most one edge, we deduce that
for any $i$, $V_{i-1}, V_i$ and $V_{i+1}$ are not in a single triangle of $\TT_0$.

On the other hand, if $d_{a_{i+1}}d_{a_i}\neq0$,
then there is a non-zero multiplication in
\[
    \Hom^\bullet(S_{j_{i-1}}[\delta_{i-1}], S_{j_{i}}[\delta_{i}])\otimes
    \Hom^\bullet(S_{j_{i}}[\delta_{i}], S_{j_{i+1}}[\delta_{i+1}])
    \to
    \Hom^\bullet(S_{j_{i-1}}[\delta_{i-1}], S_{j_{i+1}}[\delta_{i+1}]).
\]
As such a multiplication is induced from terms in the potential
(\cite[\S~A.15]{K8}), which are 3-cycles,
we deduce that $V_{i_1}, V_i$ and $V_{i+1}$ are in a single triangle of $\TT_0$.
This contradicts the fact mentioned above.
The case when $d_{a_{i}}d_{a_{i+1}}=0$ is similar.

Now we deduce that $\diff_\eta^2=0$ as required.
\end{proof}

\begin{remark}\label{rem:shift}
As we are flexible about the choice of $\delta_1$ and $t_1$,
$X_\eta$ is well-defined up to shifts. In other words, we obtain a map
\begin{gather}\label{eq:X}
    \widetilde{X}\colon \CA(\surfo)\to \per\EE_0/[1],
\end{gather}
\[\quad\qquad\eta\mapsto \widetilde{X}(\eta).\]
We will use the convention that $X_\eta$ will be a representative in the shift orbits
$\widetilde{X}(\eta)$ and the $X[\ZZ]$ denotes the shift orbit that contains $X$.
\end{remark}

\begin{example}\label{ex:0}
By construction,
$\widetilde{X}(s_i)=S_i[\ZZ]$, where the $s_i$ are the `initial' closed arcs in $\TT_0^*$
and $S_i$ are the simples in the canonical heart $\h_0$.
Let us have a look at some non-trivial case.
Take an initial triangulation of a $6$-gon as shown in the left picture in
Figure~\ref{fig:ex0}.
The Ext-quiver of $\h_0$ is as shown in Example~\ref{ex:2}.
Then we have
\begin{gather*}
    \widetilde{X}(\eta_1)=\Cone(S_1\to S_2[1])[\ZZ],\quad
    \widetilde{X}(\eta_2)=\Cone(X\to S_3[3])[\ZZ],
\end{gather*}
where $$X=\Cone(S_1[-2]\to S_3).$$
Here, the maps in the $\Cone$ are the unique maps (up to scaling) between the
corresponding objects.

\begin{figure}[t]\centering
\begin{tikzpicture}[scale=.35,rotate=-120]
\foreach \j in {1,...,6}{\draw[very thick](60*\j+30:6)to(60*\j-30:6);}
\foreach \j in {1,...,6}{\draw[NavyBlue,thick](120*\j-30:6)node{$\bullet$}to(120*\j+90:6)
    (120*\j-90:6)node{$\bullet$};}
\foreach \j in {1,...,3}{  \draw[red](30+120*\j:4)to(0,0);}
\foreach \j in {1,...,3}{  \draw(30-120*\j:4)node[white]{$\bullet$}node[red]{$\circ$};
    \draw(30-120*\j+14:2)node[red]{\text{\footnotesize{$s_\j$}}};}
\draw(0,0)node[white]{$\bullet$}node[red]{$\circ$};
\end{tikzpicture}
\qquad
\begin{tikzpicture}[xscale=-.35,yscale=.35]
\foreach \j in {1,...,6}{\draw[very thick](60*\j+30:6)to(60*\j-30:6);}
\foreach \j in {1,...,6}{\draw[NavyBlue,very thin](120*\j-30:6)node{$\bullet$}
    to(120*\j+90:6)(120*\j-90:6)node{$\bullet$};}

\draw[red,thick](30:4)to[bend right](150:4);
\draw[red,thick](0,0).. controls +(-45:2) and +(0:2) ..(0,-4.5)
    .. controls +(180:2) and +(-180:8) ..(30:4);
\draw(0,3.5)node[red]{\text{\footnotesize{$\eta_1$}}};
\draw(-1,-2)node[red]{\text{\footnotesize{$\eta_2$}}};

\foreach \j in {1,...,3}{  \draw(30+120*\j:4)node[white]{$\bullet$}node[red]{$\circ$};}
\draw(0,0)node[white]{$\bullet$}node[red]{$\circ$};
\end{tikzpicture}
  \caption{An initial triangulation and two closed arcs in a $6$-gon}
  \label{fig:ex0}
\end{figure}
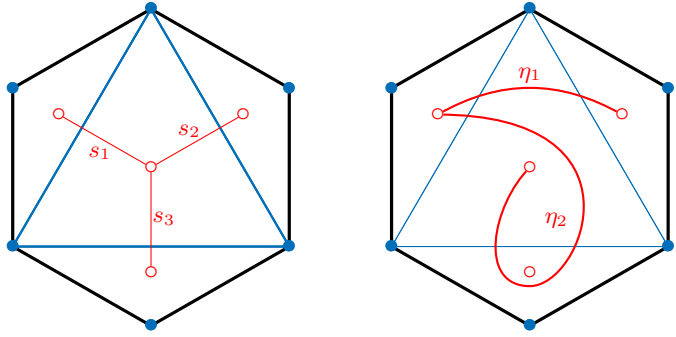
\end{example}

\begin{lemma}\label{lem:Hom}
Let $X_\eta$ be a complex associated to a closed arc $\eta$ as above.
Then
\begin{gather}\label{eq:int}
    \dim\Hom^\bullet(\Gamma^i, X_\eta)=\Int(\gamma_i,\eta),\\
    \dim\Hom^\bullet(\Gamma_0,X_\eta)=
    \sum_{i=1}^n\dim\Hom^\bullet(\Gamma^i,X_\eta)=\Int(\TT_0,\eta).
\end{gather}
\end{lemma}
\begin{proof}
First, the projective-simple duality implies
\begin{equation}\label{eq:ijkl}
    \Hom^j(\Gamma^i,S_k[l])=\delta_{ik}\cdot\delta_{jl}\cdot\k,\quad1\leq i,k\leq n;\,\forall j,l\in\ZZ.
\end{equation}
Second, the differential $\diff_\eta$ is generated by the morphisms
$\varsigma\colon S_i\to S_j[\delta]$ in \eqref{eq:EE}, which satisfy
$\Hom^\bullet(\Gamma^i,\varsigma)=0,\quad1\leq i\leq n.$
Thus the lemma follows.
\end{proof}

In particular, we have the following immediate consequence
as the $s_i$ are the only closed arcs that intersect once with $\TT$.
\begin{corollary}\label{lem:inj}
If $X_\eta[\ZZ]=S_i[\ZZ]$ for some initial closed arc $s_i\in\TT_0^*$, then $\eta=s_i$.
\end{corollary}

We will prove the following key proposition in \S~\ref{app:B}.
\begin{proposition}\label{pp:key}
Let $\eta_1$ and $\eta_2$ be two general closed arcs in $\CA(\surfo)$.
Choose any representative $X_k$ in $\widetilde{X}(\eta_k)=X_k[\ZZ]$.
Then we have
\numbers
\item If $\eta_k$ is a closed arc, i.e. is in $\cA(\surfo)$,
then  $X_{\eta_k}$ is in $\Sph(\Gamma_0)$.
\item
If $\Int(\eta_1,\eta_2)=0$, then
\begin{gather}\label{eq:=0}
    \Hom^\bullet(X_{\eta_1},X_{\eta_2})=0.
\end{gather}
\item
If $\Int(\eta_1,\eta_2)=\tfrac{1}{2}$, then
\begin{gather}\label{eq:=1}
    \dim\Hom^\bullet(X_{\eta_1},X_{\eta_2})=1.
\end{gather}
\ends
\end{proposition}

An immediate consequence of Proposition~\ref{pp:key}
and Proposition~\ref{pp:ses2} is as follows.

\begin{corollary}\label{cor:int.ts}
Let $\alpha,\beta\in\cA(\surfo)$ with $\Int(\alpha,\beta)=\tfrac{1}{2}$ and $\eta=\Bt{\alpha}(\beta)$.
Then
\begin{gather}\label{eq:stwist+}
    \widetilde{X}(\eta)=\phi_{\widetilde{X}(\alpha)}(\widetilde{X}(\beta)).
\end{gather}
\end{corollary}

\section{Braid twists versus spherical twists}\label{sec:main}
\subsection{Two twist group actions}
We start with a generalization of Corollary~\ref{cor:int.ts}.

\begin{lemma}\label{lem:actions}
For any $s\in\TT_0^*$ and $\eta\in\cA(\surfo)$, we have
\begin{gather}\label{eq:actions}
    \phi^\varepsilon_{\widetilde{X}(s)}\left( \widetilde{X}(\eta) \right)=
        \widetilde{X}\left( \operatorname{B}_s^\varepsilon(\eta) \right),
\end{gather}
where $\varepsilon\in\{\pm1\}$.
\end{lemma}
\begin{proof}
Without loss of generality, we only deal the case for $\varepsilon=1$.
Use induction on $\Int(\TT_0,\eta)$ starting with the trivial case
when $\Int(\TT_0,\eta)=1$, or equivalently, $\eta\in\TT_0^*$.
Now, for the inductive step, consider $\eta$ with $\Int(\TT_0,\eta)=m$
while the lemma holds for any $\eta'$ with $\Int(\TT_0,\eta)<m$.
Applying Lemma~\ref{lem:dcp}, we have $\eta=\Bt{\alpha}(\beta)$ for some $\alpha,\beta$
with $\Int(\alpha,\beta)=\tfrac{1}{2}$.
Twisted by $\Bt{s}$, we have $\Int(\Bt{s}(\alpha),\Bt{s}(\beta))=\tfrac{1}{2}$
and $\Bt{s}(\eta)=\Bt{\Bt{s}(\alpha)}(\Bt{s}(\beta))$.
By \eqref{eq:stwist+}, we have
\begin{gather}\label{eq:sss}
    \widetilde{X}\left( \Bt{s}(\eta) \right)=
    \phi_{ \widetilde{X}\left( \Bt{s}(\alpha) \right) }
    \left(   \widetilde{X}\left( \Bt{s}(\beta) \right)    \right).
\end{gather}
By the inductive assumption,
\begin{gather}\label{eq:ind.ass}
    \phi_{\widetilde{X}(s)}\left( \widetilde{X}(\alpha) \right)=
    \widetilde{X}\left( \Bt{s}(\alpha) \right),\quad
    \phi_{\widetilde{X}(s)}\left( \widetilde{X}(\beta) \right)=
    \widetilde{X}\left( \Bt{s}(\beta) \right).
\end{gather}
So
\[\begin{array}{rll}
    \phi_{\widetilde{X}(s)}\left( \widetilde{X}(\eta) \right)&=&
    \phi_{\widetilde{X}(s)}\left(\phi_{\widetilde{X}(\alpha)}(\widetilde{X}(\beta)) \right)\\
    &=&
    \phi_{\widetilde{X}(s)}\circ\phi_{\widetilde{X}(\alpha)}\circ\phi_{\widetilde{X}(s)}^{-1}
    \left(\phi_{\widetilde{X}(s)}(\widetilde{X}(\beta)) \right)\\
    &=&
    \phi_{ \phi_{\widetilde{X}(s)}\left( \widetilde{X}(\alpha) \right) }
    \left(\phi_{\widetilde{X}(s)}(\widetilde{X}(\beta)) \right)\\
    &=&
    \phi_{ \widetilde{X}\left( \Bt{s}(\alpha) \right) }
    \left(   \widetilde{X}\left( \Bt{s}(\beta) \right)    \right)\\
    &=&\widetilde{X}\left( \Bt{s}(\eta) \right),
\end{array}\]
where the first equality follows from \eqref{eq:stwist+},
the third equality follows from \eqref{eq:212},
the fourth equality follows from \eqref{eq:ind.ass}
and the last equality follows from \eqref{eq:sss},
which completes the proof.
\end{proof}

\begin{remark}
Let $Z_0^{\ST}=\ST(\Gamma_0)\cap\ZZ[1]$ and
$$\ST_*(\Gamma_0)=\ST(\Gamma_0)/Z_0^{\ST}\subset\Aut^\circ\D_{fd}(\Gamma_0)/\ZZ[1].$$
Note that $\ST_*(\Gamma_0)$ also acts on $\Sph(\Gamma_0)/[1]$.
By \cite[Theorem~4.4]{BQ}, $Z_0^{\ST}=1$ unless $\surf$ is a polygon,
in which case, $Z_0^{\ST}=\ZZ[n+3]$.
\end{remark}

Recall that the initial triangulation consists of closed arcs $s_i$,
whose braid twists $b_i=\Bt{s_i}$ generate $\BT(\TT_0)=\BT(\surfo)$ by Lemma~\ref{lem:gg}.
Moreover, the canonical heart $\h_0$ in $\D_{fd}(\Gamma_0)$ has simples
$S_i$ satisfying $S_i[\ZZ]=\widetilde{X}(s_i)$,
whose spherical twists $\phi_i=\phi_{S_i}$ generate $\BT(\surfo)$.

\begin{proposition}\label{pp:iota}
There is a canonical group homomorphism
\begin{gather}\label{eq:iota}
    \iota\colon\BT(\TT_0)\to\ST_*(\Gamma_0),
\end{gather}
sending the generator $b_i$ to the generator $\phi_{i}$.
\end{proposition}
\begin{proof}
Consider first the case when $\surf$ is not a polygon.
We only need to prove that, if
\begin{gather}\label{eq:b}
    b=b_{i_1}^{\varepsilon_1}\circ\cdots\circ b_{i_k}^{\varepsilon_k}
\end{gather}
equals $1$ in $\MCG(\surfo)$, for some $i_j\in\{1,\ldots,n\},\varepsilon_j\in\{\pm1\},1\leq j\leq k$
and $k\in\NN$,
then
\begin{gather}\label{eq:b=}
    \phi=\phi_{i_1}^{\varepsilon_1}\circ\cdots\circ \phi_{i_l}^{\varepsilon_l}
\end{gather}
equals $1$ in $\Aut^\circ\D_{fd}(\Gamma_0)$.

First, $b=1$ implies $b(s_i)=s_i$ for any $1\leq i\leq n$.
By (repeatedly using) Lemma~\ref{lem:actions}, we have
\[
    \widetilde{X}\left( b(s_i) \right)=
    \phi \left( \widetilde{X}(s_i) \right).
\]
Thus, $S_i[\ZZ]=\widetilde{X}\left( s_i \right)=\phi \left( S_i[\ZZ] \right)$,
i.e. $\phi(S_i)=S_i[t_i]$ for some integer $t_i$.
Since $\phi$ is an equivalence, we deduce that all $t_i$ must be the same.
Therefore $\phi=[t]$ for some integer $t$.
However, we have $\phi\in Z_0^{\ST}=1$ in this case, which implies $t=0$ and $\phi=1$
in $\Aut^\circ\D_{fd}(\Gamma_0)$, as required.

In the case when $\surfo$ is a polygon,
$b=1$ still implies $\phi=[t]$ for some $t\in\ZZ$ and
thus the proposition holds too.
\end{proof}

A consequence of the existence of $\iota$ is that
the braid twist group actions $\BT(\surfo)$ on $\cA(\surfo)$
are compatible with the spherical twist group actions $\ST_*(\Gamma_0)$ on $\Sph(\Gamma_0)/[1]$,
under the map $\widetilde{X}$ in \eqref{eq:X}.
More precisely, we have the commutative diagram below,
where the commutativity is in the sense of \eqref{eq:actions+} in the following corollary.
\begin{gather}
\begin{tikzpicture}[xscale=.6,yscale=.6]
\draw(180:3)node(o){$\cA(\surfo)$}(-3,2.2)node(b){\small{$\BT(\surfo)$}}
(0,2.5)node{$\iota$}(0,.5)node{$\widetilde{X}$};
\draw(0:3)node(a){$\Sph(\Gamma_0)/[1]$}(3,2.2)node(s){\small{$\ST_*(\Gamma_0)$}};
\draw[->,>=stealth](o)to(a);\draw[->,>=stealth](b)to(s);
\draw[->,>=stealth](-3.2,.6).. controls +(135:2) and +(45:2) ..(-3+.2,.6);
\draw[->,>=stealth](3-.2,.6).. controls +(135:2) and +(45:2) ..(3+.2,.6);
\end{tikzpicture}
\label{eq:acts}
\end{gather}

\begin{corollary}\label{cor:actions}
For any $b\in\BT(\surfo)$ and $\eta\in\cA(\surfo)$, we have
\begin{gather}
\label{eq:iota2}
    \iota(\operatorname{B}^\varepsilon_{\eta})=\phi_{\widetilde{X}(\eta)}^\varepsilon,
    \quad\varepsilon\in\{\pm1\}\\
\label{eq:actions+}
    \widetilde{X}\left( b(\eta) \right)
    =\iota(b) \left( \widetilde{X}(\eta) \right).
\end{gather}
\end{corollary}
\begin{proof}
Again, we will only deal with the case when $\varepsilon=1$.
By Proposition~\ref{pp:oa}, $\eta=b(s_j)$ for some $s_j\in\TT^*$ and $b\in\BT(\surfo)$
with the form \eqref{eq:b}.
Let $\phi$ be as in \eqref{eq:b=} and by (repeatedly using) \eqref{eq:actions},
we have $$\widetilde{X}(\eta)=\widetilde{X}(b(s_j))=\phi(\widetilde{X}(s_j))=\phi(S_j).$$
Then using formulae \eqref{eq:formulaB}, \eqref{eq:212} and the equality above
we have
\[\begin{array}{rll}
    \iota(\Bt{\eta})&=&\iota(\Bt{b(s)})\\&=&\iota\left(
    b_{i_1}^{\varepsilon_1}\circ\cdots\circ b_{i_k}^{\varepsilon_k}
    \circ \Bt{s_j} \circ
    b_{i_1}^{-\varepsilon_1}\circ\cdots\circ b_{i_k}^{-\varepsilon_k}
    \right)\\
    &=&
    \iota(b_{i_1}^{\varepsilon_1})\circ\cdots\circ \iota(b_{i_k}^{\varepsilon_k})
    \circ \iota(b_j) \circ
    \iota(b_{i_1}^{-\varepsilon_1})\circ\cdots\circ \iota(b_{i_k}^{-\varepsilon_k})\\
    &=&
    \phi_{i_1}^{\varepsilon_1}\circ\cdots\circ \phi_{i_k}^{\varepsilon_k}
    \circ \phi_j \circ
    \phi_{i_1}^{-\varepsilon_1}\circ\cdots\circ \phi_{i_k}^{-\varepsilon_k}    \\
    &=&\phi \circ \phi_j \circ \phi^{-1}\\
    &=&\phi_{\phi(S_j)}
    =\phi_{\widetilde{X}(\eta)},
\end{array}\]
i.e. \eqref{eq:iota2}.
A similar calculation 
gives \eqref{eq:actions+}, as the generalization of \eqref{eq:actions}.
\end{proof}

When specifying $b=\Bt{s}^\varepsilon$ in \eqref{eq:actions+} and using \eqref{eq:iota2},
we see that \eqref{eq:actions} holds for any $s,\eta\in\cA(\surfo)$.

\begin{corollary}
\eqref{eq:actions} holds for any $s,\eta\in\cA(\surfo)$.
\end{corollary}

Now, we are ready to prove the main theorem of this paper.
\subsection{The main result}
We start to show that $\widetilde{X}$ is bijective.

\begin{theorem}\label{thm:bijection}
The map $\widetilde{X}$ in \eqref{eq:X} induces a bijection
\[
    \widetilde{X}\colon\cA(\surfo)\xrightarrow{\text{1-1}}\Sph(\Gamma_0)/[1].
\]
\end{theorem}
\begin{proof}
First we prove the injectivity.
Suppose $\widetilde{X}(\eta)=\widetilde{X}(\eta')$ for $\eta,\eta'\in\cA(\surfo)$.
Let $\eta=b(s_i)$ for some $b\in\BT(\surfo)$ and initial closed arc $s_i\in\T_0^*$.
Then by \eqref{eq:actions+} we have
\[
    S_i[\ZZ]=\widetilde{X}(s_i)=\widetilde{X}(b^{-1}(\eta))=
    \iota(b)^{-1}\left(\widetilde{X}(\eta)\right)=
    \iota(b)^{-1}\left(\widetilde{X}(\eta')\right)=
    \widetilde{X}(b^{-1}(\eta')).
\]
By Corollary~\ref{lem:inj}, $s_i=b^{-1}(\eta')$ or $\eta=\eta'$ as required.

Second, we prove the surjectivity.
Let $\eta$ be a closed arc in $\cA(\surfo)$ and
$\widetilde{X}(\eta)=X_\eta[\ZZ]$ for some representative $X_\eta$.
We only need to show that $X_\eta$ is in $\Sph(\Gamma_0)$.
Use induction on $I=\Int(\TT_0,\eta)$.
If $I=1$, then $\eta$ is some $s_i\in\TT_0$ and $X_\eta=S_i[\delta]$
for some integer $\delta$, which is in $\Sph(\Gamma_0)$.
Now suppose that the claim is true for $I\leq r$ for some $r\geq1$
and consider the case when $I=r+1$.
Apply Lemma~\ref{lem:dcp}, we find $\alpha$ and $\beta$ with
$\Int(\alpha,\beta)=\frac{1}{2}$ and \eqref{eq:tt}.
By Corollary~\ref{cor:int.ts}, we have representatives $X_\alpha$
and $X_\beta$ with \eqref{eq:stwist}.
By the inductive assumption, we know that $X_\alpha$ and $X_\beta$
are in $\Sph(\Gamma_0)$.
On the other hand, we have
$\phi_{X_{\alpha}}\in\ST(\Gamma_0)$ by \eqref{eq:212}
and the theorem follows from \eqref{eq:sph=st}.
\end{proof}

We proceed to show that the bijectivity above implies isomorphism
between twisted groups.

\begin{theorem}\label{thm:main}
Let $\surf$ be an unpunctured marked surface and $\TT_0$ a triangulation of $\surfo$
such that the corresponding FST'quiver has no double arrows.
Then there is a canonical isomorphism
\begin{gather}\label{eq:main}
    \iota\colon\BT(\TT_0)\to\ST(\Gamma_0),
\end{gather}
sending the generator $b_i$ to the generator $\phi_{i}$, where $\Gamma_0$
is the Ginzburg dg algebra associated to $\TT_0$.
\end{theorem}
\begin{proof}
When $\surf$ is a polygon, this follows from \cite{KS} and \cite{ST}.
Now suppose $\surf$ is not a polygon.
We first prove the case for the initial triangulation $\TT_0$
(whose FST quiver has no double arrows).
Then $\ST_*(\Gamma_0)=\ST(\Gamma_0)$.
In this case, we have the surjective homomorphism $\iota$ in \eqref{eq:iota}
and only need to show that it is injective.

Let $b\in\BT(\surfo)$ with $\iota(b)=1$ in $\ST(\Gamma_0)$.
By \eqref{eq:actions+}, we have
\[
    \widetilde{X}\left( b(\eta) \right)
    =\iota(b) \left( \widetilde{X}(\eta) \right)= \widetilde{X}(\eta),
\]
which implies $b(\eta)=\eta$ by Theorem~\ref{thm:bijection},
for any closed arc $\eta$.
By \eqref{eq:commutes}, this implies $b\circ\Bt{\eta}=\Bt{\eta}\circ b$
and thus $b$ is the center $Z_0^{\BT}$ of $\BT(\surfo)$.
But $Z_0^{\BT}=1$ in this case.
So $b=1$ and $\iota$ is injective.
\end{proof}

\begin{remark}\label{rem:general}
We can generalize Theorem~\ref{thm:main} to any triangulations $\TT\in\EGp(\surfo)$, i.e. as Theorem~1.
This follows by a standard induction, on the number of flips from $\TT_0$ to $\TT$;
so we only need to prove the case when $\TT$ is a flip of $\TT_0$.

On one hand, $\TT^*$ and $\TT_0^*$ are related by a Whitehead move as in Figure~\ref{fig:WH}.
Thus, the standard generators of $\BT(\TT)$ are conjugates of standard generators of $\BT(\TT_0)$.
It is straightforward to write down the formula of the conjugates.
On the other hand, this is also true for $\ST(\Gamma_\TT)$ and $\ST(\Gamma_0)$.
Namely,
\begin{itemize}
\item by \cite{KY}, there is a (canonical) derived equivalence
$$\Psi\colon\D_{fd}(\Gamma_\TT)\cong\D_{fd}(\Gamma_0),$$
such that the canonical heart $\h_{\Gamma_\TT}$ becomes
a tilt $\h'$ (cf. \cite[Definition~3.7]{KQ}) of the canonical heart $\h_0$;
\item \cite[Proposition~5.4]{KQ} provides a formula for how simples change under tilting
(i.e. each simple in $\h'$ is a twist or a shift of some simple in $\h$);
\item then we deduce that under the induced isomorphism $\Psi_*\colon\ST(\Gamma_\TT)\cong\ST(\Gamma_0)$,
the standard generators of $\ST(\Gamma_\TT)$ become the conjugates of
the standard generators of $\ST(\Gamma_0)$.
\end{itemize}
By comparing the two formulae of the conjugates, we deduce that
\eqref{eq:main} implies \eqref{1}.

We will use the same trick again in \S~\ref{sec:Kro} to prove the special cases in Remark~\ref{rem:sp},
which completes the generalization from Theorem~\ref{thm:main} to Theorem~1.
\end{remark}

\section{Special cases}\label{sec:Kro}
In this section, we first deal with the two special cases in Remark~\ref{rem:sp}.
Then we discuss the affine $\widetilde{A}$ case in more detail.
\subsection{The Kronecker case}
We first discuss the special case I) in Remark~\ref{rem:sp}.
Note that in case I), all triangulations of $\surf$ or $\surfo$ look the same,
cf. Figure~\ref{fig:Kro}.
Choose any triangulation $\TT_0$ of $\surfo$ as the initial triangulation.
Keep all the notations as above.

The dynamic of proof here is the reverse compared with the previous cases:
we will show the relation between the twist groups first;
then the relations between closed arcs and spherical objects.

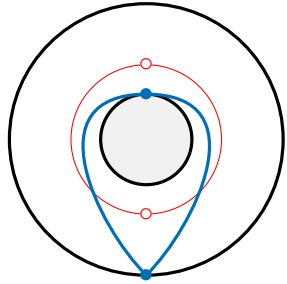
\begin{figure}[ht]\centering
\begin{tikzpicture}[scale=.3]
\draw[red](0,0)circle(3.3);
  \draw[very thick,NavyBlue](0,8)node{};
  \draw[very thick](0,0)circle(6);  \draw[very thick, fill=gray!11](0,0)circle(2);
  \draw[very thick,,NavyBlue] (0,-6).. controls +(135:3) and +(180:5) ..(0,2)
    .. controls +(0:5) and +(45:3) ..(0,-6);
  \draw[very thick,NavyBlue](0,-6)node{$\bullet$}(0,2)node{$\bullet$};
  \draw(0,3.3)node[white] {$\bullet$} node[red]{$\circ$};\draw(0,-3.3)node[white] {$\bullet$} node[red]{$\circ$};
\end{tikzpicture}
  \caption{The Kronecker case}
  \label{fig:Kro}
\end{figure}

First, we claim that \eqref{eq:main} also holds in this case.

\begin{proposition}\label{pp:main1}
Let $\surf$ be an annulus with two marked points and $\TT_0$ a triangulation of $\surfo$.
There is a canonical isomorphism
\begin{gather}\label{eq:main2}
    \iota\colon\BT(\TT_0)\to\ST(\Gamma_0),
\end{gather}
sending the generator $b_i$ to the generator $\phi_{i}$, where $\Gamma_0$
is the Ginzburg dg algebra
associated to $\TT_0$.
\end{proposition}
\begin{proof}
Consider an annulus $\surfo'$ with triangulation $\TT_0'$ (cf. left picture in Figure~\ref{fig:A12}),
whose FST quiver is the affine quiver $Q'$ of type $\widetilde{A_{1,2}}$:
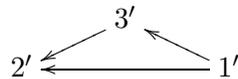
\[\xymatrix@R=.3pc{
    &3'\ar@{<-}[dr]\\ 2'\ar@{<-}[ur]\ar@{<-}[rr]&&1'
}\]
We can choose another triangulation $\TT'$, as shown in the right picture in Figure~\ref{fig:A12},
whose FST quiver is
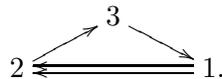
\[\xymatrix@R=.5pc{
    &3\ar[dr]\\ 2\ar[ur]\ar@{<-}@<.3ex>[rr]\ar@{<-}@<-.3ex>[rr]&&1.
}\]
By Remark~\ref{rem:general}, we have \eqref{1} for $\TT'$.
On the other hand, we have the following two facts:
\begin{itemize}
\item the subcategory $\D_0$ of $\D_{fd}({\Gamma'})$ generated by $X'_{1}$ and $X'_{2}$
is equivalent to the 3-CY category for a Kronecker quiver, where $X'_i$ is the spherical object
corresponding to $s_i'$;
\item there is a subsurface $\mathbf{Y}_\Tri$ of $\surfo'$, with inherited triangulation from $\TT'$
(whose dual consists of $s_1'$ and $s_2'$),
that is isomorphic to any triangulation of an annulus with two marked points.
\end{itemize}
Therefore, by identifying $\D_{fd}({\Gamma_0})$ with $\D_0$
and $\surfo$ with $\mathbf{Y}_\Tri$, we have
$$\ST(\Gamma_0)\cong\<\phi_{X_1'},\phi_{X_2'}\>\cong\<\Bt{s_1'},\Bt{s_2'}\>\cong\BT(\TT_0),$$
which implies the proposition.

\end{proof}

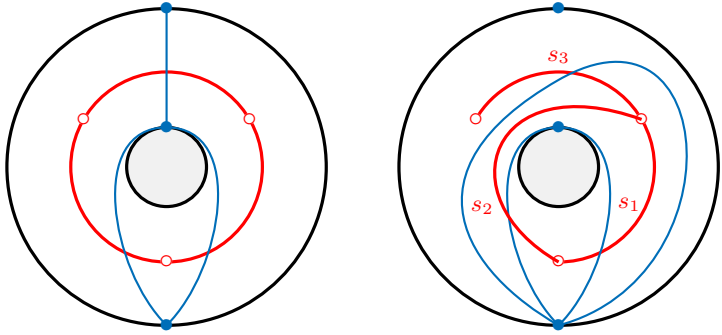
\begin{figure}[t]\centering
\begin{tikzpicture}[scale=.35]
\draw[very thick, red](0,0)circle(3.6);
  \draw[very thick,NavyBlue](0,8)node{};
  \draw[very thick](0,0)circle(6);  \draw[very thick,fill=gray!11](0,0)circle(1.5);
  \draw[NavyBlue,thick] (0,-6).. controls +(135:3) and +(180:3) ..(0,1.5)
    .. controls +(0:3) and +(45:3) ..(0,-6);
  \draw[NavyBlue,thick](0,-6)node{$\bullet$}(0,1.5)node{$\bullet$}to(0,6)node{$\bullet$};
    \foreach \j in {1,...,3}{
  \draw(-90+120*\j:3.6)node[white] {$\bullet$} node[red]{$\circ$};}

  \draw(90-120*1:3)node[red]{\text{\footnotesize{}}}
  (90-120*2:3)node[red]{\text{\footnotesize{}}}
  (90-120*3-10:3.5)node[above,red]{\text{\footnotesize{}}} ;
\end{tikzpicture}
\qquad
\begin{tikzpicture}[scale=.35]
\draw[very thick, red](-90:3.6)arc(-90:150:3.6);
  \draw[very thick,NavyBlue](0,8)node{};
  \draw[very thick](0,0)circle(6);  \draw[very thick,fill=gray!11](0,0)circle(1.5);
  \draw[NavyBlue,thick] (0,-6).. controls +(135:3) and +(180:3) ..(0,1.5)
    .. controls +(0:3) and +(45:3) ..(0,-6);
  \draw[very thick,NavyBlue](0,-6)node{$\bullet$}(0,1.5)node{$\bullet$}(0,6)node{$\bullet$};
  \draw[NavyBlue,thick] (0,-6).. controls +(15:7) and +(30:7) ..(0,3.3)
    .. controls +(0-150:7) and +(150:4) ..(0,-6);
    \foreach \j in {1,...,3}{
  \draw(-90+120*\j:3.6)node[white] {$\bullet$} node[red]{$\circ$};}

  \draw[very thick, red](-90:3.6).. controls +(150:5) and +(160:6) ..(30:3.6)
  (-28:3)node{\text{\footnotesize{$s_1$}}}(28+180:3.25)node{\text{\footnotesize{$s_2$}}}
  (90:3.5)node[above,red]{\text{\footnotesize{$s_3$}}};
\end{tikzpicture}
  \caption{$\widetilde{A_{1,2}}$ case: $\TT'$ on the left and $\TT_0'$ on the right}
  \label{fig:A12}
\end{figure}


\subsection{The one marked point torus case}
In this section, we give the analogue of Proposition~\ref{pp:main1}
for the special case II) in Remark~\ref{rem:sp}.
The proof is almost the same, by considering
a torus with one boundary component and two marked points on it for instance.
\begin{proposition}\label{pp:main2}
Let $\surf$ be a torus with one marked point and $\TT_0$ a triangulation of $\surfo$.
There is a canonical isomorphism
\begin{gather}\label{eq:main3}
    \iota\colon\BT(\TT_0)\to\ST(\Gamma_0),
\end{gather}
sending the generator $b_i$ to the generator $\phi_{i}$, where $\Gamma_0$
is the Ginzburg dg algebra
associated to $\TT_0$.
\end{proposition}

\subsection{Example: annulus case}
When $\surf$ is an annulus, Theorem~\ref{thm:main},
(together with Proposition~\ref{pp:main1}) can be stated as follows.

\begin{theorem}\label{thm:mainA}
Let $\surf$ be an annulus and $\TT$ be a triangulation of $\surfo$
with associated Ginzburg dg algebra $\Gamma_\TT$.
Suppose there are $p$ and $q$ marked points on the two boundary components of $\surf$,
respectively.
Then the spherical twist group $\ST(\Gamma_\TT)$ is (canonically) isomorphic
to the braid group $\Br(\widetilde{A_{p,q}})$ of affine $\widetilde{A_{p,q}}$.
\end{theorem}
\begin{proof}
The case $p=q=1$ is Proposition~\ref{pp:main1},
noticing that the braid group $\Br(\widetilde{A_{1,1}})$ is a rank $2$ free group.
The other case follows from Theorem~\ref{thm:main},
noticing that $\BT(\surfo)$
is (canonically) isomorphic to $\Br(\widetilde{A_{p,q}})$
by the geometric description of the affine braid group in \cite{KP}.
\end{proof}

\section{On the space of stability conditions}\label{sec:sss}
\subsection{Stability conditions}
First recall Bridgeland's notion of stability conditions.

\begin{definition}[cf. \cite{BS}]\label{def:stab}
A \emph{stability condition} $\sigma = (Z,\hua{P})$ on
a triangulated category $\hua{D}$ consists of
a group homomorphism (\emph{the central charge}) $Z:K(\hua{D}) \to \kong{C}$ and
full additive subcategories $\hua{P}(\varphi) \subset \hua{D}$
for each $\varphi \in \kong{R}$, satisfying the following axioms:
\begin{itemize}
\item if $0 \neq E \in \hua{P}(\varphi)$
then $Z(E) = m(E) \exp(\varphi  \pi \mathbf{i} )$ for some $m(E) \in \kong{R}_{>0}$;
\item $\hua{P}(\varphi+1)=\hua{P}(\varphi)[1]$, for all
$\varphi \in \kong{R}$;
\item if $\varphi_1>\varphi_2$ and $A_i \in \hua{P}(\varphi_i)$
then $\Hom_{\hua{D}}(A_1,A_2)=0$;
\item for each nonzero object $E \in \hua{D}$ there is a finite sequence of real numbers
$$\varphi_1 > \varphi_2 > ... > \varphi_m$$
and a collection of triangles (the Harder-Narashimhan filtration)
\begin{equation}\label{eq:filt}\xymatrix@C=0.8pc@R=1.4pc{
  0=E_0 \ar[rr] && E_1 \ar[dl] \ar[rr] &&   E_2 \ar[dl] \ar[rr] && ... \
  \ar[rr] && E_{m-1} \ar[rr] && E_m=E \ar[dl] \\
  & A_1 \ar@{-->}[ul]  && A_2 \ar@{-->}[ul] &&  && && A_m \ar@{-->}[ul]
}\end{equation}
with $A_j \in \hua{P}(\varphi_j)$ for all $j$.
\end{itemize}
Let $I$ be an interval in $\kong{R}$ and define $\hua{P}(I)$
to be the subcategory generated by $\{\hua{P}(\varphi)\mid\varphi\in I\}$.
The heart of a stability condition $\sigma = (Z,\hua{P})$ on $\hua{D}$ is $\hua{P}[0,1)$.
\end{definition}
An important result by Bridgeland is that
all stability conditions on a triangulated category $\D$
form a space $\Stab(\D)$ that has the structure of a complex manifold.
We are interested in the stability conditions on the 3-CY category $\D_{fd}(\Gamma)$
for a Ginzburg dg algebra $\Gamma$ arising from quivers with potential.
Note that for the stability conditions on $\D_{fd}(\Gamma)$
whose heart is the canonical heart $\zero$ form a half open half closed $n$-cell $U(\zero)$
in $\Stab\D_{fd}(\Gamma)$ (see \cite{Q2}).
Denote by $\Stap\D_{fd}(\Gamma)$ the connected component of $\Stab\D_{fd}(\Gamma)$ that contains
$U(\zero)$.

\subsection{Quadratic differentials}
Recall that $\surf$ is a marked surface with initial triangulation $\TT_0$, associated Ginzburg dg algebra $\Gamma_0$ and $\Aut^\circ\D_{fd}(\Gamma_0)$ is defined as in \eqref{eq:Aut0}.
Denote by $\Quad_\heartsuit(\surf)$ the moduli space of quadratic differentials on $\surf$,
in the sense of \cite[\S6]{BS}.
The main result there is as follows.

\begin{theorem}\cite[Theorem~1.2]{BS}
As complex manifolds,
$    \Stap\D_{fd}(\Gamma_0)/\Aut^\circ\cong\Quad_\heartsuit(\surf).$
\end{theorem}

For our purpose, we prefer to deal with the space $\Quad(\surf)$ of quadratic differentials
on a fixed marked surface $\surf$ instead of the moduli space.
These two spaces of quadratic differentials differ
by the symmetry of the marked mapping class group $\MMCG(\surf)$, i.e.
$$\Quad_\heartsuit(\surf)=\Quad(\surf)/\MMCG(\surf).$$
Here, $\MMCG(\surf)$ of a marked surface $\surf$
is the group of isotopy classes of (orientation preserving) homeomorphisms of $\surf$,
where all homeomorphisms and isotopies are required to fix the set $\M$ of marked points as a set.

By \cite[Theorem~9.9]{BS}, there is the short exact sequence
\begin{equation}\label{eq:ses}
    1\to\ST(\Gamma_0)\to\Aut^\circ\D_{fd}(\Gamma_0)\to\MMCG(\surf)\to1
\end{equation}
and the theorem above can be alternatively stated as:
$    \Stap\D_{fd}(\Gamma_0)/\ST\cong\Quad(\surf).$
Thus there is a short exact sequence
\begin{equation}\label{eq:pi}
    1\to  \pi_1\Stap\D_{fd}(\Gamma_0) \to \pi_1 \Quad(\surf) \xrightarrow{\pi} \ST(\Gamma_0) \to 1.
\end{equation}

\subsection{On the Contractibility}
In this subsection, let $\surf$ be an annulus with $p$ and $q$ marked points on its boundary components respectively.

Suppose first $p\neq q$.
It is straightforward to calculate $\MMCG(\surf)$ in this case:
it is generated by the two rotations along the two boundary components.
More precisely, $\MCG(\surf)$ is the infinite cyclic group generated by
the Dehn twist $\Dehn{C}$ along the only (up to isotopy)
non-trivial simple closed curve in $\surf$.
The two rotations are the $p$-th and $q$-th roots of $\Dehn{C}$,
denoted by $r_0$ and $r_1$, respectively.
Then $\MMCG(\surf)$ is the abelian group with generators $r_0$ and $r_1$ and
with relation $r_0^p=r_1^q$, which fits into the following short exact sequence
\[
    1\to\ZZ\<r_0\>\to\MMCG(\surf)\to\ZZ_q\<r_1\>\to1.
\]
Besides $\underline{\xi}=r_0\cdot r_1$ is the universal rotation
that corresponds to $[1]$.

Next, as shown in \cite[\S12.3]{BS},
\begin{equation}\label{eq:pi1}
    \Quad_\heartsuit(\surf)\cong\Conf^n(\kong{C}^*)/\ZZ_q,
\end{equation}
where $\Conf^n(\kong{C}^*)$ denotes the configuration space of $n$ distinct points
in $\kong{C}^*$ and $\ZZ_q$ acts by multiplication by a $q$-th root of unity.
By the description of $\Br(\widetilde{A_{p,q}})$ in \cite{KP},
there is short exact sequence
\begin{equation}\label{eq:ses2}
    1\to\Br(\widetilde{A_{p,q}})\to\pi_1\Conf^n(\kong{C}^*)\to\ZZ\to1.
\end{equation}
As $\Quad_\heartsuit(\surf)$ consists of differentials of the form
\[
    \Theta(z)=\prod_{i=1}^n (z-z_i)\frac{\diff z^{\otimes2}}{z^{p+2}},\quad
    z_i\in\kong{C}^*,\quad z_i\neq z_j
\]
and considered modulo the action of $\kong{C}$ rescaling $z$.
Note that $z_i$ corresponds to the decorating points in $\surfo$,
the rotation $r_q$ becomes the $\kong{Z}_q$ symmetry at the origin
and the rotation $r_p$ becomes the $\kong{Z}_p$ symmetry at the infinity.
Thus, combining the short exact sequences above
and the calculation of fundamental groups of spaces in \eqref{eq:pi1},
we have the commutative diagram \eqref{eq:comm},
which implies the dashed short exact sequence.
\begin{equation}\label{eq:comm}
\xymatrix@C=3pc{&1\ar[d]&1\ar@{-->}[d]\\
    &\Br(\widetilde{A_{p,q}})\ar@{=}[r]\ar[d]&\Br(\widetilde{A_{p,q}})\ar@{-->}[d]\ar[r]
        &1\ar[d]\\
    1\ar[r]&\pi_1\Conf^n(\kong{C}^*)\ar[d]\ar[r]&\pi_1\Quad_\heartsuit(\surf)\ar[r]\ar@{-->}[d]&
        \ZZ_q\ar[r]\ar@{=}[d]&1\\
    1\ar[r]&\ZZ\ar[d]\ar[r]&\MMCG(\surf)\ar@{-->}[d]\ar[r]&\ZZ_q\ar[r]&1\\
    &1&1
}.\end{equation}
Therefore we have $\pi_1\Quad(\surf)=\Br(\widetilde{A_{p,q}})$ and hence
$\pi_1\Quad(\surf)\cong\ST(\Gamma_0)$ by Theorem~\ref{thm:mainA}.
Further, by examining the generators, we deduce that the surjective map $\pi$ in \eqref{eq:pi}
gives the isomorphism above.
Thus, $\Stap\D_{fd}(\Gamma_0)$ is simply connected.

In the case when $p=q$, $\MMCG(\surfo)$ contains one more $\kong{Z}_2$ symmetry.
In the same way, we will have $\pi_1\Quad(\surf)=\Br(\widetilde{A_{p,q}})$ and simply connectedness.

\begin{theorem}\label{thm:ss}
Let $\surf$ be an annulus (without punctures) and
$\D_{fd}(\Gamma_0)$ be the 3-CY category associated to some triangulation of $\surf$.
Then $\Stap\D_{fd}(\Gamma_0)$ is the universal cover of $\Conf^n(\kong{C}^*)$.
\end{theorem}

By \cite[Theorem~2.7]{CP}, the universal cover of $\Conf^n(\kong{C}^*)$ is contractible.
So we have:

\begin{corollary}
$\Stap\D_{fd}(\Gamma_0)$ is contractible.
\end{corollary}

\section{Proof of Proposition~\ref{pp:key}}\label{app:B}
\subsection{Preparation}
See \cite[Appendix~A]{QZ2} for the details of homological algebra calculations for the string model in \S~\ref{sec:sm}.
The key results are the following two.

\begin{proposition}\cite[Corollary A.9]{QZ2}\label{pp:qz2}
Let $\eta_1, \eta_2$ be two closed arcs in $\surfo$ that share an endpoint.
Fix orientations of them and suppose that they share the starting endpoint $Z\in\Tri$.
Then there is a unique non-zero homomorphism
$\zeta^Z_{12}\in\Hom^\bullet(X_\alpha,X_\beta)$ induced by $Z$.
Moreover, suppose there is another closed arc $\eta_3$ starting at $Z$,
such that $\eta_1,\eta_3,\eta_2$ are in a clockwise order at $Z$.
Then $\zeta^Z_{12}$ is the composition of $\zeta^Z_{13}$ with $\zeta^Z_{23}$.
\end{proposition}

\begin{proposition}\cite[Proposition~A.11]{QZ2}\label{pp:ses2}
Let $\alpha, \beta$ and $\eta$ be three closed arcs in $\overline{\cA}(\surfo)$ such that
at least one of them is in $\cA(\surfo)$.
Moreover, we require that $\alpha,\eta,\beta$ are in a clockwise order
to form a contractible triangle in $\surfo$.
Then there are representatives $X_?$ in $\widetilde{X}(?)$ for $?=\alpha,\beta,\eta$ such that
there is a non-trivial triangle
\begin{gather}\label{eq:tri}
    X_\beta \to X_\eta \to X_\alpha \to X_\beta[1],
\end{gather}
where the homomorphisms are of the form in Proposition~\ref{pp:qz2}.
\end{proposition}

\begin{remark}\label{rem:l}
Note that, in the setting of Lemma~\ref{lem:dcp},
the line segment $l$, from $Z_0$ to some point $Y$ in $\eta$ (cf. Figure~\ref{fig:line}),
plays an important role.
We will say $\eta$ decomposes into $\alpha$ and $\beta$ w.r.t. $l$.

Also note that the condition $\Int(\eta_1,\eta_2)=\frac{1}{2}$ in $3^\circ$ forces that
$\eta_1,\eta_2$ are closed arcs in $\cA(\surfo)$.
\end{remark}

\subsection{The first induction}
Use double induction, the first on
\begin{gather}\label{eq:I}
    I=\Int(\TT_0,\eta_1)+\Int(\TT_0,\eta_2).
\end{gather}
The starting case is when $I=2$.
Then both $\eta_1$ and $\eta_2$ are in $\TT_0^*$,
since the only general closed arcs that have exactly one intersection with $\TT_0$
are the arcs in $\TT_0^*$.
It is straightforward to check the proposition in this case.
Now suppose that the proposition holds for any $(\eta_1,\eta_2)$
with $I\leq r$ and consider the case when $I=r+1$.

First, let us prove $1^\circ$ for $X_{\eta}$ (where $\eta=\eta_1$ or $\eta=\eta_2$ in $\cA(\surfo)$).
Apply Lemma~\ref{lem:dcp} to decompose $\eta$ into $\alpha$ and $\beta$ in $\cA(\surfo)$ with $\Int(\alpha,\beta)=\tfrac{1}{2}$ (cf. Figure~\ref{fig:line}).
Then by Proposition~\ref{pp:ses2}, there is a non-trivial triangle \eqref{eq:tri}
By the inductive assumption, the proposition holds for $\alpha$ and $\beta$.
Then $\alpha,\beta\in\cA(\surfo)$ implies $X_{\alpha},X_{\beta}\in\Sph(\Gamma_0)$
and $\Int(\alpha,\beta)=\tfrac{1}{2}$ implies
\[
    \dim\Hom^\bullet(X_{\alpha},X_{\beta})=1.
\]
Hence \eqref{eq:tri} implies
\begin{gather}\label{eq:stwist}
    X_{\eta}=\phi_{X_\alpha}(X_\beta)=\phi^{-1}_{X_\beta}(X_\alpha),
\end{gather}
and thus $X_{\eta}$ is also in $\Sph(\Gamma_0)$.

\subsection{The second induction}
Next, we prove $2^\circ$ and $3^\circ$.
Use the second induction on
\[\min\{\Int(\TT_0,\eta_1),\Int(\TT_0,\eta_2)\}.\]
Without loss of generality, suppose that
\begin{gather}\label{eq:leq}
    \Int(\TT_0,\eta_1)\leq\Int(\TT_0,\eta_2).
\end{gather}

The starting case is when $\Int(\TT_0,\eta_1)=1$,
which implies that $\eta_1=s_i$ for some $i$.
Note that we have $\Int(\TT_0,\eta_2)>1$.
Applying Lemma~\ref{lem:dcp}
to decompose $\eta_2$ into $\alpha$ and $\beta$, w.r.t. some decorating point $Z_0$.
As above, we get a non-trivial triangle \eqref{eq:tri} by Proposition~\ref{pp:ses2}.
There are two cases.

\begin{description}
\item[Case i] If $Z_0$ is not an endpoint $\eta_1$.\end{description}
Then the inductive assumption
holds for $(\eta_1,\alpha)$ and $(\eta_1,\beta)$.
For $2^\circ$,
we have $\Int(\eta_1,\alpha)=0=\Int(\eta_1,\beta)$ and hence
\begin{gather}\label{eq:00}
    \Hom^\bullet(X_{\eta_1},X_{\alpha})=0
    =\Hom^\bullet(X_{\eta_1},X_{\beta}).
\end{gather}
Applying $\Hom(X_{\eta_1},?)$
to triangle \eqref{eq:tri}, we obtain \eqref{eq:=0}.
For $3^\circ$, we have
$$\{\Int(\eta_1,\alpha),\Int(\eta_1,\beta)\}=\{\tfrac{1}{2},0\},$$
and hence one of
$\Hom^\bullet(X_{\eta_1},X_{\alpha})$ and $\Hom^\bullet(X_{\eta_1},X_{\beta})$
is zero while the other one has dimension one.
Applying $\Hom(X_{\eta_1},?)$
to triangle \eqref{eq:tri}, we obtain \eqref{eq:=1}.
\begin{description}
\item[Case ii] If $Z_0$ is an endpoint $\eta_1$.
\end{description}
\begin{figure}[hb]\centering
\begin{tikzpicture}[yscale=.5,xscale=.6,rotate=180]
\draw[red, thick,->-=.51,>=stealth](-2,3)to[bend right](2,3);
\draw[red,thick](-2,2.5)to[bend right=7](0,1.5)(2,2.5)to[bend right=7](0,1.5);

\draw[](-2,1.9)node{$\alpha$}(2,1.9)node{$\beta$};
\draw[NavyBlue, thick](-2,0)node{$\bullet$}to(2,0)node{$\bullet$}to(0,4)node{$\bullet$}to(-2,0)
    (0,1.5)node[white]{$\bullet$}node[red]{$\circ$} (0,2.5)node[below,black]{$\eta_2$};
\end{tikzpicture}
\quad
\begin{tikzpicture}[yscale=.5,xscale=.6,rotate=0]
\draw[red, thick,-<-=.4,>=stealth](-2,1.5)to[bend right](2,1.5);
\draw[red,thick](-2,2.5)to[bend right=7](0,1.5)(2,2.5)to[bend right=7](0,1.5);

\draw[](2,3)node{$\alpha$}(-2,3)node{$\beta$};
\draw[NavyBlue, thick](-2,0)node{$\bullet$}to(2,0)node{$\bullet$}to(0,4)node{$\bullet$}to(-2,0)
    (0,1.5)node[white]{$\bullet$}node[red]{$\circ$} (0,.9)node[below,black]{$\eta_2$};
\end{tikzpicture}
\quad
\begin{tikzpicture}[yscale=.5,xscale=.6,rotate=180]
\draw[red, thick,->-=.4,>=stealth](-1.5,-1.5).. controls +(80:5) and +(100:5) ..(1.5,-1.5);
\draw[red,thick](-1,-1.5)to[bend right=-7](0,1.5)(1,-1.5)to[bend right=7](0,1.5);

\draw[](-1,-1.5)node[above]{$\alpha$}(1,-1.5)node[above]{$\beta$};
\draw[NavyBlue, thick](-2,0)node{$\bullet$}to(2,0)node{$\bullet$}to(0,4)node{$\bullet$}to(-2,0)
    (0,1.5)node[white]{$\bullet$}node[red]{$\circ$} (0,2.3)node[below,black]{$\eta_2$};
\end{tikzpicture}

\caption{The three cases for possible position of $\alpha$ and $\beta$}
\label{fig:3cases}
\end{figure}
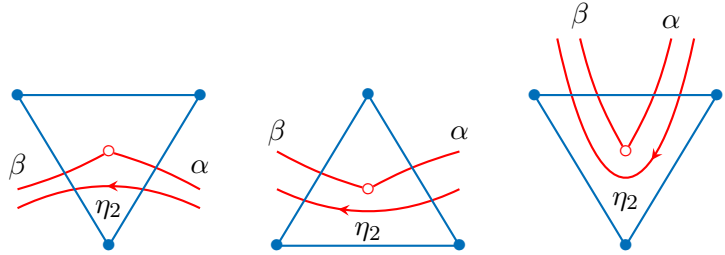

For $2^\circ$, we have
$\Int(\eta_1,\alpha)=\tfrac{1}{2}=\Int(\eta_1,\beta)$
and thus (by the inductive assumption)
\begin{gather}\label{eq:homs1}
    \dim\Hom^\bullet(X_{\eta_1},X_{\alpha})=1
    =\dim\Hom^\bullet(X_{\eta_1},X_{\beta}).
\end{gather}
There are three cases (as shown in Figure~\ref{fig:3cases})
for the possible positions of $\alpha$ and $\beta$
in the triangle $\Lambda_0$ that contains $Z_0$. 
Since $\eta_1$ does not intersect $\eta_2$, the line segments of $\eta_1,\alpha,\beta$ in $\Lambda_0$
are in a clockwise order.
By Proposition~\ref{pp:qz2},
when applying $\Hom(X_{\eta_1},?)$ to the triangle \eqref{eq:tri},
there will be an isomorphism 
\begin{gather}\label{eq:iso}
    \Hom^t(X_{\eta_1},\alpha)\xrightarrow{\simeq}\Hom^t(X_{\eta_1},\beta[1])
\end{gather}
in the long exact sequence for some $t\in\ZZ$,
which implies \eqref{eq:=0} by \eqref{eq:homs1}.

For $3^\circ$,
without loss of generality, suppose that
$\alpha$ and $\eta_1$ do not intersect in $\surfo-\Tri$
but share both endpoints,
and $\Int(\eta_1,\beta)=\tfrac{1}{2}$.
Then $\dim\Hom^\bullet(X_{\eta_1},X_{\beta})=1$.
If $\alpha$ is in $\TT_0^*$, then we have $\alpha=s_i=\eta_1$.
Applying the inductive assumption to $(\eta_1, \beta)$, we have
$\eta_2=\Bt{\eta_1}(\beta)$ and $X_{\eta_2}=\phi_{X_{\eta_1}}(X_\beta)$.
Then
$$    \dim\Hom^\bullet(X_{\eta_1},X_{\eta_2})=\dim\Hom^\bullet(X_{\eta_1},X_{\beta})=1,$$
as required.
Otherwise, apply Lemma~\ref{lem:dcp} to decompose $\alpha$ into closed arcs $\alpha'$ and $\beta'$.
By applying the inductive assumption to $(\eta_1,\alpha')$ and $(\eta_1,\beta')$,
we deduce that
\[
    \dim\Hom^\bullet(X_{\eta_1},{X}_{\alpha'})=1=
    \dim\Hom^\bullet(X_{\eta_1},{X}_{\beta'}).
\]
and hence $\dim\Hom^\bullet(X_{\eta_1},X_{\alpha})$ is $0$ or $2$.
Moreover, Proposition~\ref{pp:qz2} implies an isomorphism between a subspace
of $\Hom^\bullet(X_{\eta_1},\alpha)$ and $\Hom^t(X_{\eta_1},\beta[1])$
(cf. \eqref{eq:iso}), which implies
\[
    \dim\Hom^\bullet(X_{\eta_1},X_{\eta_2})\leq
        \dim\Hom^\bullet(X_{\eta_1},X_{\alpha})+
        \dim\Hom^\bullet(X_{\eta_1},X_{\beta})-2=1.
\]
One the other hand,
\[
    \dim\Hom^\bullet(X_{\eta_1},X_{\eta_2})\equiv
        \dim\Hom^\bullet(X_{\eta_1},X_{\alpha})+
        \dim\Hom^\bullet(X_{\eta_1},X_{\beta})\equiv1(\mod 2).
\]
Therefore \eqref{eq:=1} holds as required.

\subsection{Inductive step of the second induction}
To finish the proof, we only need to show that
if $2^\circ$ and $3^\circ$ hold for $I\leq r$ or $I=r+1$ with $\Int(\TT_0,\eta_1)\leq r_1$,
then they hold for $I=r+1$ with $\Int(\TT_0,\eta_1)=r_1+1$
(recall that $I$ is defined in \eqref{eq:I} and we assume \eqref{eq:leq}).

Apply Lemma~\ref{lem:dcp} to decompose $\eta=\eta_1$ into $\alpha, \beta$ w.r.t.
some decorating point $Z_0$ and some line segment $l$ (see Figure~\ref{fig:line}).
\begin{description}
\item[Case i] The line segment $l$
does not intersect $\eta_2$ in $\surfo-\Tri$.\end{description}
Then neither $\alpha$ nor $\beta$ intersect $\eta_2$ in $\surfo-\Tri$.
Since $\eta_1$ and $\eta_2$ don't share two endpoints,
without loss of generality, suppose that
the common endpoint of $\eta_1$ and $\beta$ is not an endpoint of $\eta_2$.
Consider
\[\eta_1'=\Bt{\beta}(\eta_1)=\alpha\quad
    \text{and}\quad \eta_2'=\Bt{\beta}(\eta_2).
\]
See Figure~\ref{fig:BT beta} for the two possibilities, where $Z'$ and $Z''$ could coincide.
As in \eqref{eq:stwist}, we have
\[
    X_{\eta_1}=\phi^{-1}_{X_\beta}(X_\alpha)=\phi^{-1}_{X_\beta}(X_{\eta_1'})
    \quad\text{and}\quad
    X_{\eta_2}=\phi^{-1}_{X_\beta}(X_{\eta_2'}),
\]
which implies
\begin{gather}\label{eq:stwist2}
    \Hom^\bullet(X_{\eta_1},X_{\eta_2})\simeq
    \Hom^\bullet(X_{\eta_1'},X_{\eta_2'})
\end{gather}
Moreover, we have
\begin{gather*}
    \Int(\TT_0,\eta_1')=\Int(\TT_0,\alpha)=
    \Int(\TT_0,\eta_1)-\Int(\TT_0,\beta);\\
    \Int(\TT_0,\eta_2')=\Int(\TT_0,\Bt{\beta}(\eta_2))\leq
    \Int(\TT_0,\eta_2)+\Int(\TT_0,\beta).
\end{gather*}
Thus $2^\circ$ or $3^\circ$ hold for $(\eta_1',\eta_2')$
by the inductive assumption, which implies that they also hold for
$(\eta_1,\eta_2)$ by \eqref{eq:stwist2}.

\begin{description}
\item[Case ii]
The line segment $l$ intersects $\eta_2$.\end{description}
Let $Y'$ be their nearest intersection to $Z_0$.
Then we can decompose $\eta_2$ to $\alpha$ and $\beta$,
using the line segment $l'=YZ_0(\subset l)$ as in Lemma~\ref{lem:dcp}.
There is a small difference here, that $Z_0$ might be an endpoint of $\eta_2$,
so $\alpha$ and $\beta$ are in $\CA(\surfo)$ (i.e. they might be L-arc instead of closed arc).
Since $Z_0$ is not an endpoint of $\eta_1$, we deduce that
\[\begin{array}{c}
    \frac{1}{2}\geq\Int(\eta_1,\eta_2)=\Int(\eta_1,\alpha)+\Int(\eta_1,\beta).
\end{array}\]
As $\Int(\TT_0,\alpha)+\Int(\TT_0,\beta)=\Int(\TT_0,\eta_2)$,
the inductive assumption applies to $(\eta_1,\alpha)$ and $(\eta_1,\beta)$.
Then $\dim\Hom^\bullet(X_{\eta_1},{X}_{\alpha})$ and $\dim\Hom^\bullet(X_{\eta_1},{X}_{\beta})$
are both zero (for $2^\circ$) and are $\{0,1\}$ for $3^\circ$.
Either way, we will have $\Hom^\bullet(X_{\eta_1},X_{\eta_2})=2\Int(\eta_1,\eta_2)$
as required.

\begin{figure}[t]\centering
\begin{tikzpicture}[scale=.85]
\draw[red,thick](0,0).. controls +(90:1) and +(-90:1) ..(3,2);
\draw[red,thick](6,0).. controls +(180:1) and +(0:1) ..(3,2);

\draw[red,thick](0,0).. controls +(90:.7) and +(210:1) ..(3,1)
    .. controls +(30:1) and +(180:1) ..(6,0);

\draw[red,thick](6,.5).. controls +(180:1) and +(0:1) ..(4.5,3);

\draw(0,0)node[white] {$\bullet$} node[red]{$\circ$}node[left]{}
     (3,2)node[white] {$\bullet$} node[red]{$\circ$}node[above]{}
     (6,0)node[white] {$\bullet$} node[red]{$\circ$}node[right]{$Z'$}
     (6,0.5)node[white] {$\bullet$} node[red]{$\circ$}node[right]{$Z''$}
     (4.5,3)node[white] {$\bullet$} node[red]{$\circ$}node[right]{}
     (5.5,2)node[right]{$\eta_2=\eta_2'$}
     (4.8,1)node {$\alpha$}(1.5,1.3)node {$\beta$}(3,0.9)node[below]
     {$\eta_1$};
\end{tikzpicture}
\quad
\begin{tikzpicture}[scale=.85]
\draw[red,thick](0,0).. controls +(90:1) and +(-90:1) ..(3,2);
\draw[red,thick](6,0).. controls +(180:1) and +(0:1) ..(3,2);

\draw[red,thick](0,0).. controls +(90:.7) and +(210:1) ..(3,1)
    .. controls +(30:1) and +(180:1) ..(6,0);

\draw[red,thick](6,.5).. controls +(120:1.5) and +(30:1) ..(3,2);
\draw[red,thick](6,.5).. controls +(90:4) and +(90:4) ..(0,0);
\draw(3,3.7)node{$\eta_2'$}(4,2)node[above]{$\eta_2$};

\draw(0,0)node[white] {$\bullet$} node[red]{$\circ$}node[left]{}
     (3,2)node[white] {$\bullet$} node[red]{$\circ$}node[above]{}
     (6,0)node[white] {$\bullet$} node[red]{$\circ$}node[right]{$Z'$}
     (6,0.5)node[white] {$\bullet$} node[red]{$\circ$}node[right]{$Z''$}
     (4.8,1)node {$\alpha$}(1.5,1.3)node {$\beta$}(3,0.9)node[below]
     {$\eta_1$};
\end{tikzpicture}
\caption{$\eta_1'=\alpha$ and $\eta_2'$}
\label{fig:BT beta}
\end{figure}
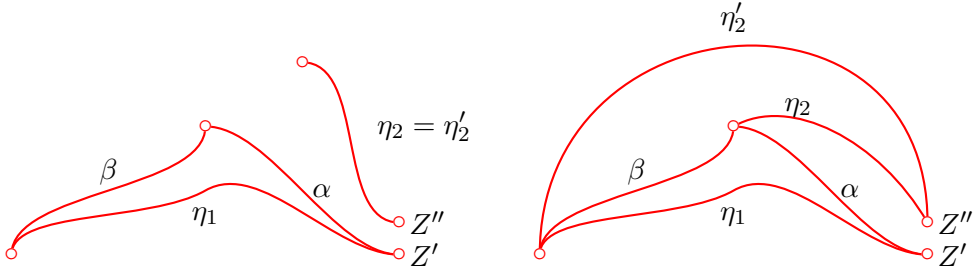

\section{Further studies}\label{sec:ss}
\subsection{Algebraic twist group of quivers with potential}\label{app:ATG}
Let $(Q,W)$ be a rigid quiver with potential such that
there is no double arrow in $Q$ and $W$ is the sum of some cycles in $Q$.

\begin{definition}\label{def:bg}
The \emph{algebraic twist group} $\AT(Q,W)$ of such a quiver with potential $(Q,W)$
is the group generated by $\{t_i\mid i\in Q_0\}$ subject to the relations
\numbers
  \item $t_i t_j= t_j t_i$ if there is no arrow between $i$ and $j$ in $Q$,
  \item $t_i t_j t_i = t_j t_i t_j$ if there is exactly one arrow between $i$ and $j$ in $Q$,
  \item $R_i=R_j$ for any $i,j$ (cyclic relation),
  if there is a cycle $Y\colon1 \to 2 \to \cdots \to m \to 1$ in $Q$ (or a term in $W$ by definition),
  where $R_i=t_i t_{i+1} \cdots t_{2m+i-3}$ with convention $k=m+k$ here.
\ends
\end{definition}

First, we show that any cyclic relations in Definition~\ref{def:bg},
that correspond to the same cycle $Y$, are equivalent to each other.

\begin{lemma}\label{lem:rel}
Let $m\geq3$ and suppose that $t_1,t_2, \cdots, t_m$ satisfy the relations
\begin{gather}\label{eq:relss}
\begin{cases}
    t_j t_i t_j=t_i t_j t_i,  &  |j-i|=1\text{ or } \{i,j\}=\{1,m\},\\
    t_i t_j=t_j t_i,  &  \text{otherwise}.
\end{cases}
\end{gather}
Let $k=m+k$ and $R_i=t_i t_{i+1} \cdots t_{2m+i-3}$.
Then the relation $R_1=R_2$ is equivalent to $R_1=R_i$ for any $3\leq i\leq m$.
\end{lemma}
\begin{proof}
By the relations in \eqref{eq:relss}, it is straightforward to check the following
\[\begin{array}{rll}
  t_i R_1=R_1 t_{i-2}, \quad &i=2,\cdots,m-1.\\
  t_i R_{i+1}=R_i t_{i-2}, \quad &i=3,\cdots, m.
\end{array}
\]
Then we have
\[\begin{array}{rll}
    R_1=R_i & \Longleftrightarrow\; R_1 t_{i-2}= R_i t_{i-2}\\
    &\Longleftrightarrow\; t_1 R_1= t_i R_{i+1}\\
    &\Longleftrightarrow\; R_1=R_{i+1}
\end{array}
\]
for any $i=2,\cdots, m-1$, which implies the lemma.
\end{proof}

A consequence of Lemma~\ref{lem:rel} is
\[
    R_i=R_j \;\Longleftrightarrow\; R_k=R_l
\]
provided $i\neq j$ and $k\neq l$.

The following result was originally in \cite{KQ1} for type $A$ and $D$,
which is also independently obtained by Grant-Marsh for all Dynkin types.
\begin{proposition}\label{pp:GM}
If $(Q,W)$ is mutation-equivalent to a Dynkin diagram $\underline{Q}$,
then the algebraic twist group $\AT(Q,W)$ is isomorphic to
the corresponding braid group $\Br({\underline{Q}})$.
\end{proposition}

\begin{remark}
We believe that the proposition above also holds for the affine Dynkin case,
as long as $Q$ does not have double arrows.
The point is, one should be able to define an algebraic twist group
for a (good) quiver with potential,
which provides a presentation
of the corresponding spherical twist group (or/and Dehn twist group).
\end{remark}

\subsection{Intersection formulae}\label{sec:int}
Interpreting the intersection formulae between open (resp. closed) arcs
as dimension of $\Hom$ (resp. $\Ext$) play a crucial role
in many proofs (e.g.
in \cite{KS} 
and in \cite{QZ}). 
We have the following conjecture.

\begin{conjecture}\label{con0}
Let $\alpha,\beta\in\cA(\surfo)$.
We have
\begin{gather}\label{eq:con0}
    \dim\Hom^\bullet(\widetilde{X}(\alpha),\widetilde{X}(\beta))=2\Int(\alpha,\beta).
\end{gather}
\end{conjecture}

Moreover, we have another conjectured formula.
\begin{conjecture}\label{con1}
Denote by $\oAp(\surfo)$ the set of open arcs that appear in triangulations in $\EGp(\surfo)$.
Then there is a map $\rho\colon\oAp(\surfo)\to\per\Gamma_0$,
such that any $\eta\in\cA(\surfo)$,
\begin{gather}\label{eq:dimHom+}
    \dim\Hom^\bullet(\rho(\gamma),\widetilde{X}(\eta))=\Int(\gamma,\eta).
\end{gather}
\end{conjecture}

We will prove these two intersection formulae in \cite{QZ2}.




\end{document}